\documentclass[11pt]{article}

\usepackage[latin1]{inputenc}
\usepackage[T1]{fontenc}
\usepackage{lmodern}

\usepackage{ifmtarg}
\usepackage{xcolor}
\usepackage{cancel}
\usepackage{amsmath}
\usepackage{amsthm}
\usepackage{amssymb}
\usepackage[heavycircles]{stmaryrd}
\usepackage{bm}
\usepackage{accents}
\usepackage[%
  showonlyrefs=false,%
  showmanualtags=true%
]{mathtools}

\usepackage{calc}
\usepackage{tikz}
\usepackage{comment}
\usepackage[a4paper, margin=2.5cm]{geometry}

\usepackage[english]{babel}
\usepackage[%
  unicode=true,%
  pdfusetitle=true,%
  bookmarks=true,%
  pdftex%
]{hyperref}
\pdfoutput=1

\allowdisplaybreaks[3]

\DeclareMathOperator*{\esssup}{ess\,sup}
\DeclareMathOperator*{\essinf}{ess\,inf}

\DeclareMathOperator*{\ulim}{lim\uparrow}

\newcommand{\N}{\mathbb{N}}

\newcommand{\R}{\mathbb{R}}

\renewcommand{\leq}{\leqslant}
\renewcommand{\geq}{\geqslant}

\makeatletter
\let\old@underbar\underbar
\renewcommand*{\underbar}[1]{%
  \relax\ifmmode
    \underaccent{\bar}{#1}
  \else
    \old@underbar{#1}
  \fi%
}
\makeatother

\makeatletter
\newcommand*{\suchas}[1][]{\, \@ifnotmtarg{#1}{\expandafter\csname#1\endcsname}\vert \,}
\makeatother
\def\given{\suchas}

\newlength{\thmbefspace}
\newlength{\thmaftspace}
\newtheoremstyle{thmstyle}%
  {\thmbefspace}%
  {\thmaftspace}%
  {\itshape}%
  {0pt}%
  {\bfseries}%
  {}%
  {0pt}%
  {\thmname{#1}\thmnumber{~#2}\thmnote{~(\boldmath#3\unboldmath)} --- }
\newtheoremstyle{it}% name % cf. thmtest.tex of AMSLaTeX
  {3pt}%      Space above
  {3pt}%      Space below
  {\itshape}% Body font
  {}%         Indent amount (empty = no indent,
  {\bf}% Thm head font
  { -- }%        Punctuation after thm head
  { }%        Space after thm head: " " = normal interword space;
  {}%         Thm head spec (can be left empty, meaning `normal')
\newtheoremstyle{rm}% name % cf. thmtest.tex of AMSLaTeX
  {3pt}%      Space above
  {3pt}%      Space below
  {}%         Body font
  {}%         Indent amount (empty = no indent,
  {\bf}% Thm head font
  { -- }%        Punctuation after thm head
  { }%        Space after thm head: " " = normal interword space;
  {}%         Thm head spec (can be left empty, meaning `normal')
\theoremstyle{it}
\newtheorem{theorem}{Theorem}[section]%
\newtheorem{corollary}[theorem]{Corollary}%
\newtheorem{proposition}[theorem]{Proposition}%
\newtheorem{lemma}[theorem]{Lemma}%
\newtheorem{definition}[theorem]{Definition}%
\theoremstyle{rm}%
\newtheorem{example}[theorem]{Example}%
\newtheorem{remark}[theorem]{Remark}%
\newtheorem{notation}[theorem]{Notation}%
\newcommand{\F}{\mathcal{F}}
\newcommand{\T}{\mathcal{T}}
\renewcommand{\S}{\mathcal{S}}

\makeatletter
\let\thebibliography@old\thebibliography
\let\endthebibliography@old\endthebibliography

\renewenvironment{thebibliography}[1]{%
  \let\refname@old\refname%
  \let\bibname@old\bibname%
  \renewcommand{\refname}{\refname@old\phantomsection}%
  \renewcommand{\bibname}{\refname@old\phantomsection}%
  \thebibliography@old{#1}%
  \renewcommand{\bibname}{\bibname@old}%
  \renewcommand{\refname}{\refname@old}%
}{%
  \endthebibliography@old
  \renewcommand{\refname}{\refname@old}%
  \renewcommand{\bibname}{\refname@old}%
}
\makeatother

\newcommand*{\aname}[1]{\textsc{#1}}
\makeatletter
\newcommand*{\atitle}[2][]{%
  \@ifmtarg{#1}{
    \emph{#2}%
  }{%
    \ifthenelse{%
      \equal{#1}{en}%
    }{%
      \def\myLanguageArgument{english}%
    }{%
      \def\myLanguageArgument{#1}%
    }%
    \foreignlanguage{\myLanguageArgument}{\emph{#2}}%
  }%
}
\makeatother

\title{\texorpdfstring{\aname{Dynkin}}{Dynkin} games in a general framework}
\author{\texorpdfstring{%
Magdalena \aname{Kobylanski} \footnote{ (LAMA - UMR 8050) -- Universit\'e Paris~Est \newline{}\href{mailto:magdalena.kobylanski@univ-mlv.fr}{\nolinkurl{magdalena.kobylanski@univ-mlv.fr}}}%
\and{}%
  Marie-Claire \aname{Quenez}\footnote{Laboratoire de Probabilit\`es et Mod\`eles Al\'eatoires (L.P.M.A.) -- Universit\'e Denis \aname{Diderot} -- Paris~7 / Inria\newline{}\href{mailto:quenez@math.jussieu.fr}{\nolinkurl{quenez@math.jussieu.fr}}}%
\and{}%
  Marc \aname{Roger de Campagnolle}\footnote{Laboratoire de Probabilit\`es et Mod\`eles Al\'eatoires  (L.P.M.A.) -- Universit\'e Denis \aname{Diderot} -- Paris~7\newline{}\href{mailto:marcrdc@math.jussieu.fr}{\nolinkurl{marcrdc@math.jussieu.fr}}}%
}{%
  Magdalena KOBYLANSKI \&{} Marie-Claire QUENEZ \&{} Marc ROGER de CAMPAGNOLLE%
}}
\date{\today}

\begin{document}
\maketitle%
\setlength{\parskip}{\medskipamount}%
\begin{abstract}
We revisit the Dynkin game problem in a general framework and relax some assumptions. The payoffs and the criterion are expressed in terms of  families of random variables indexed by stopping times. We construct two nonnegative supermartingales  families $J$ and $J'$ whose finitness is equivalent to the Mokobodski's condition. Under some weak right-regularity assumption on the payoff families, the game is shown to be fair and $J-J'$ is shown to be the common value function. 
Existence of saddle points is derived under some weak additional assumptions.
All the results are written in terms of random variables and are proven by using  only classical results of  probability theory.
\end{abstract}\bigskip
\section*{Introduction}

In this paper, the Dynkin game problem is revisited in the general framework of families of random variables indexed by stopping times.
The criterion is given for each pair of stopping times $(\tau, \sigma)$ by the random variable
\[I_{0} \left( \tau, \sigma \right):=E[\xi ( \tau ) {\bf 1}_{\left\{ \tau \leq \sigma \right\}} + \zeta ( \sigma ) {\bf 1}_{\left\{ \sigma < \tau \right\}}].\]
In full generality, 
$\xi(\tau)$ and $\zeta(\sigma)$ are random variables, respectively $\mathcal{F}_{\tau}$\,-measurable and $\mathcal{F}_{\sigma}$\,-measurable. 
Studying the Dynkin game problem consists in proving that, under suitable conditions, the game is fair that is, 
\[\inf_{\sigma}\sup_{\tau}I_{0} \left( \tau, \sigma \right)=\sup_{\tau}\inf_{\sigma}I_{0} \left( \tau, \sigma \right),\]
in characterizing this common value function,  and finally in proving the existence of saddle points.

The Dynkin game problem has been largely studied in the literature in the framework of processes. The first results (cf.   \cite{B1}, \cite{B2}, \cite{A}) were obtained under the Mokobodski condition, which stipulates that the payoffs are separated by the difference of two nonegative a.s. finite supermartingales. Existence of a value for the game is then obtained by supposing only  right-upper semi continuity on $(\xi_t)$ and $(-\zeta_t)$ (cf. \cite{A}). Yet, checking Mokobodski's condition appears as  a difficult question.   
We stress on that this approach relies on sophisticated results of the General Theory of Processes and Optimal Stopping Theory.

Note that, in the literature, there are other works related to Dynkin Games, for instance  \cite{CK}, \cite{HHO}. We refer to section \ref{section.moko} for details. 

Recently, Kobylanski and Quenez  in \cite{KQ}  have revisited the optimal stopping problem in the case of a reward given by a {\em family} of {\em random variables} indexed by stopping times. This notion is very general and includes the case of processes as a particular case.
This setup has appeared as relevant and appropriate as it allows, from the one hand, to release some hypotheses made on the reward, and on the second hand, to make simpler  proofs using simpler tools.%

In the present work, the setup of families of random variables indexed by stopping times allows to solve the Dynkin game problem under very weak assumptions by using only classical tools of Probability Theory. 

%Note also that in El Karoui's course, more precisely in the first part, the setup of families of random variables, which is of course more general than the one of processes, is underlined as having much interest. However, in the second part, the setup of processes is seen as necessary. Indeed, the tools used in the proofs rely on strong results of the General Theory of Processes such as, between many others, some fine {\em aggregation} results and the {\em Merten's decomposition} of supermartingales which corresponds to a generalization of Doob Meyer's decomposition.  We underline that the method is highly sophisticated.

 \vspace{0,2cm}

The paper is organised as follows. In section 1, we introduce the Dynkin game problem. In our set up, the payoffs (or rewards) 
 $\xi:= \left( \xi(\theta), \theta\in \T\right)$ and $\zeta:= \left( \zeta(\theta), \theta\in \T\right)$ are given by families of integrable 
 random variables indexed by stopping times, which satisfy some natural compatibility conditions. In section 2, from $\xi$ and 
 $\zeta$, we construct two $[0, + \infty]$-valued  supermartingale families $J= \left( J(\theta), \theta\in \T\right)$ and 
 $J'= \left( J'(\theta), \theta\in \T\right)$. These two families  satisfy $J= \mathcal{R}(J'+\xi)$ and $J'=\mathcal{R}(J-\zeta)$ 
 where $\mathcal{R}$ is the Snell envelope operator.
In the case of processes, this construction is classical (see for example \cite{A}) and is done under Mokobodski's condition.
In the present work, we do not need any condition of this type in order to define $J$ and $J'$. When $J$ is a.s. finite 
(or equivalently $J'$ is  a.s. finite), the difference $J-J'$ is well defined and is proven to satisfy $\xi \le J-J'\le \zeta$. 
 In section 3, under the previous assumption, we prove that $J-J'$ is the value of the game, first when optimal stopping times 
 do exist, and second  when $\xi$ and $-\zeta$ are right-upper semicontinuous
along stopping times in expectation.
In section 4, under some additional  assumptions, we derive the existence of saddle points. In section 5, we show that  condition 
$J$ a.s. finite not only implies but is also equivalent to  Mokobodski's condition, and some complementary results are provided. 
At last, comes the Appendix. In Appendix A, we first briefly recall some results of Kobylanski and Quenez (2012)  \cite{KQ} which 
are used in this paper. Second, we provide two lemmas used to prove the existence result. Finally, in Appendix B, we apply our results to the case of processes.

We point out that, whereas in the previous works, the proof of the existence of saddle points relies on some highly sophisticated tools 
of the General Theory of Processes, the one given in this paper does not require any  over prerequisite  than those given in the 
Appendix  and is only done by using classical probability results.  Also, condition $J(0) < + \infty$ may be seen as easier to check than  
Mokobodski's condition. Finally, some conditions on the payoffs are relaxed in comparison with previous results (cf. \cite{A}).

\vspace{0,5cm}

We introduce some notation. Let ${\mathbb F}=(\Omega, \F, (\F_t)_{{0\le t\le T}},P)$ be a probability space equipped with a filtration 
$(\F_t)_{{0\le t\le T}}$ satisfying the usual conditions of right continuity and augmentation by the null sets  of ${\cal F}= {\cal F}_T$. 
We suppose that ${\cal F}_0$ contains only sets of probability $0$ or $1$.  The time horizon is a fixed constant  $T$ in $]0,\infty[$.
We denote by $\T$ the collection of stopping times of
${\mathbb{F} }$ with values in $[0 , T]$.  More generally, for any
stopping time $S$, we denote by $\T_{S}$  (resp.  $\T_{S^+}$) the class of stopping times
$\theta\in \T$ with $\theta\geq S$ a.s.\, (resp. $\theta>S $ a.s. on $\{S<T\}$ and $\theta=T$ a.s. on $\{S=T\}$).

We also define $\T_{[S, S^{'}]}$ the set of $\theta\in \T$ with $S \leq \theta \leq S^{'}$ a.s. 

We use the following notation: for real valued  random variables $X$ and $X_n$, $n\in$ $\N$,  ``$X_n\uparrow X$'' stands for 
``the sequence $(X_n)_{n \in \N}$ is nondecreasing and converges to $X$ a.s.''.

%\begin{remark}
%By convention, the positivity property of a random variable means that it takes its values in $\overline \R^+$.

%Also, note that it is possible to define a family associated with a given process.
%More precisely, let $(\phi_t)$ be a nonnegative progressive process. Then, the family $\{\,\phi(\theta), \TO,\}$ defined by $\phi(\theta)=\phi_\theta$ is admissible.
%\end{remark}

%\begin{remark} Note that since the value function is a supermartingale family, equality $E[v(S)]= E[ v( \theta_{*})]$ is equivalent to the fact that the family $\{\, v(\theta), \theta \in T_{S, \theta_{*}} \,\}$ is a martingale family (A DeFINIR)
%\end{remark}

\section{Dynkin games}

In this section, we present the Dynkin game problem in the framework of families of random variables indexed by stopping times. We first introduce some notation and definitions.
\subsection{Definitions and notation}

\begin{definition}\label{admi}
A  family of $\bar \R$-valued random variables $\left( \phi(\theta), \theta\in \T\right)$ is said to be {\em admissible} if it satisfies the following conditions 

\hspace{1cm}\begin{tabular}{ll}
1)& for all
$\theta\in \T$\,, $\phi(\theta)$ is an $\F_\theta$-measurable random variable (r.v.),\\
 2)& for all
$\theta,\theta'\in \T$\,, $\phi(\theta)=\phi(\theta')$ a.s. on
$\{\theta=\theta'\}$. \\
% 3)& $E [\esssup_{\theta\in
%T_0}\phi(\theta)^p]<\infty$\, for some $p>1$.
 \end{tabular}

\end{definition}
In the sequel, such a family $\left( \phi(\theta), \theta\in \T\right)$ is identified with the map $\phi$ : $\theta$ 
$\mapsto$ $\phi(\theta)$ from $\T$ into the set of random variables. 

\begin{remark}
The notion of admissible families includes the case of processes as a particular case. 
Indeed, if $(\phi_t)_{t\in \R_+}$ is a progressive process, the family of random variables 
$\bar \phi=(\bar \phi(\theta), \theta\in \T)$ defined by $\bar \phi(\theta)=\phi_\theta$ is admissible. 
\end{remark}

Let us introduce some notation and definitions.
\begin{notation} The set of admissible families is denoted by $\mathcal A$.\\
The relation \;$\ge$\; is defined on $\mathcal A$ in the following way.\\For $\phi, \phi' \in \mathcal{A}$,   $\phi \geq \phi'$ if, for each $\theta$ $\in$ $\mathcal{T}$\,, $\phi (\theta) \geq \phi'(\theta)$ a.s.
\\
The relations \;$\leq$\; and \;$=$ \; on $\mathcal{A}$ are defined in the same way.
 \\
 We denote by $0$, the family $\phi$ in $\mathcal{A}$ such that $\phi(\theta)=0$ a.s. for each $\theta\in \mathcal{T}$.
 \\
A family $\phi \in \mathcal{A}$ is \emph{non negative} if $\phi\ge0$. 
\\
A sequence $(\phi_n)_{n\in \N}$ in $\mathcal{A}$ is \emph{non decreasing} if $\phi_n\le \phi_{n+1}$ for each $n\in \N$.
\\
For $x\in \R$, $x^+=\max(x,0)$ and $x^-=-\min(x,0)$.
\\
For $\phi \in \mathcal{A}$,    $\phi^+$ denotes the family  $\displaystyle \left((\phi(\theta))^+, \theta\in \mathcal{T}\right)$,  
and  $\phi^-$ denotes the family  $\displaystyle \left((\phi(\theta))^-, \theta\in \mathcal{T}\right)$.\\
A family $\phi \in \mathcal{A}$ is said to be \emph{integrable} if, for each $\theta$ $\in$ $\mathcal{T}$, $\phi (\theta)$ is integrable.
\end{notation}
We define the following subsets of $\mathcal{A}$:
 \[ \mathcal{S}=\left\{ \phi \in \mathcal{A}, \; E\left[\esssup_{\theta\in \mathcal{T}} |\phi(\theta)|\right]<+\infty \right\}\]
\[ \mathcal{S}^+ =\left\{   \phi \in \mathcal{A}, \; \phi ^+\in \mathcal{S}\right\}, \qquad \mathcal{S}^- =\left\{   \phi \in \mathcal{A}, \;  \phi ^-\in \mathcal{S}\right\}.\]

 \begin{definition}\label{def.mart}
 An admissible family $\phi=(\phi(\theta) , \; \theta \in \mathcal{T})$, such that $\phi^- $ is integrable, is said to be a {\em supermartingale family} (resp. a {\em martingale family}) if  for any 
$\theta, \theta'$ $ \in$ $\mathcal{T}$ such that $\theta \geq \theta'$ a.s., 
\[
E[\phi(\theta) \, |\,{\cal F}_{ \theta' }] \leq \phi(\theta') \quad
\mbox{a.s.,}  \mbox{ (resp.} \quad E[\phi(\theta) \, |\,{\cal F}_{ \theta' }] = \phi(\theta') \quad
\mbox{a.s.).}
\]
\end{definition}

\begin{remark} Let $\tau$, $\sigma$ $\in$ $\mathcal{T}$ such that $\tau \leq \sigma$ a.s. and $\phi \in \mathcal{A}$. 
One can easily show that the family $\phi$ restricted to $\mathcal{T}_{[\tau, \sigma]}$\,, that is $\left(\phi(\theta), \,\theta \in \mathcal{T}_{[\tau, \sigma]} \right)$, 
is a martingale family if  and only if  $\Big( \phi \big((\theta \vee \tau) \wedge \sigma \big),\, \theta \in \mathcal{T} \Big)$ is a martingale family.
\end{remark}

Let us now recall some notation and definitions related to the optimal control problem in the framework 
of admissible families, studied in \cite{KQ}. Let $\phi= (\phi(\theta),\theta\in \T)$ be an admissible family in $\mathcal{S}^-$, called {\em reward} family.
For each $\theta \in \T$, the value function at time $\theta$ is defined by
\begin{equation*}\label{vss}
v(\theta):=\esssup_{\tau \in \T_{\theta}} E[\phi(\tau) \, |\,\F_{\theta}].
\end{equation*}
The family of random variables $v= (v(\theta),\theta\in \T)$, called {\em value function family}, is clearly admissible.
%where $T_{S^+}$ is the class of stopping times
%$\theta\in T_0$ with $\theta>S $ a.s. on $\{S<T\}$ and $\theta=T$ a.s. on $\{S=T\}$.
%Note that $\vp(\theta)= \phi (T)$ a.s. on $\{S=T\}$.

We now introduce the Snell envelope operator $\mathcal{R}$, defined in the framework of admissible families.
\begin{definition} \label{definitionSnell} For each family $\phi$ $\in$  $\mathcal{A}$,  the smallest supermartingale family greater or equal to $\phi$, is called, when it exists, the {\em Snell envelope family} of $\phi$, and is denoted by $\mathcal{R}(\phi)$.
\end{definition}

For each $\phi\in \mathcal{S}^-$, the value function family $v$ associated to $\phi$ can be shown to 
be equal to the Snell envelope family of $\phi$, that is  $v=\mathcal{R}(\phi)$. The Snell envelope operator 
$\mathcal{R}$ is thus well defined on $\mathcal{S}^-$, and valued in the set of supermartingale families.

Note that in \cite{KQ}, the optimal stopping problem is solved for a \emph{nonnegative} $\phi$ $\in$ $\mathcal{A}$. 
Yet, by translation, all the results do apply to $\phi\in \mathcal{S}^-$. This translation argument is detailed in 
Appendix \ref{subsec.opt}, as well as  the main results of \cite{KQ}.

We now introduce the Dynkin game problem.
\subsection{The Dynkin game problem}

Throughout the paper, $\xi = (\xi(\theta) , \; \theta \in \mathcal{T})$ and $\zeta= (\zeta(\theta) , \; \theta \in \mathcal{T})$ 
are two integrable admissible families such that $\xi \in \mathcal{S}^-$ and $\zeta\in \mathcal{S}^+$. We suppose that 
$\xi ( T ) = \zeta (T) = 0$ a.s. Actually, this last condition is not a restriction (see Remark \ref{re} below).

We consider the classical Dynkin game with two players. The rule of the game is as follows. Each of the players has to choose 
a stopping time, denoted by $\tau$ for the first player and $\sigma$ for the second one. The game stops at $\tau \wedge \sigma$.
 On $\{ \tau \leq \sigma \}$, the second player pays the amount 
$\xi(\tau)$ to the first one and, on $\{ \sigma < \tau \}$, the first player pays the amount $-\zeta(\sigma)$ to the second one. 
In other words, at time $\tau \wedge \sigma$, the first player receives the amount 
$X=\xi ( \tau ) {\bf 1}_{\left\{ \tau \leq \sigma \right\}} + \zeta ( \sigma ) {\bf 1}_{\left\{ \sigma < \tau \right\}}$ 
and the second one receives $-X$. The criterion at time $0$ for the strategy $(\tau, \sigma)$ is defined by the expectation at time 0 of $X$, namely,
\[I_0 \left( \tau, \sigma \right):=E[\xi ( \tau ) {\bf 1}_{\left\{ \tau \leq \sigma \right\}} + \zeta ( \sigma ) {\bf 1}_{\left\{ \sigma < \tau \right\}}].\]

At time $0$, the goal of the first (resp. second player) is to maximize (resp. minimize) this criterion. 
Now, the players are unwilling to risk. Thus, the first one wants to find a strategy $\tau$ which maximizes the quantity
$\inf_{\sigma} I_{0} \left( \tau, \sigma \right)$. His value function at time $0$ is given by 
$\underbar{V} ( 0)  \coloneqq \adjustlimits \sup_{\tau \in \mathcal{T}} \inf_{\sigma \in \mathcal{T}} I_{0} \left( \tau, \sigma \right)$.

The second one wants to find a strategy $\sigma$ which minimizes the quantity $\sup_{\tau} I_{0} \left( \tau, \sigma \right)$. His value 
function at time $0$ is given by
$
\bar{V} ( 0 ) \coloneqq \adjustlimits \inf_{\sigma \in \mathcal{T}_{0}} \sup_{\tau \in \mathcal{T}} I_{0} \left( \tau, \sigma \right). \label{vdessous0}
$

We now make the problem dynamic. For each $\theta$ $\in$ $\mathcal{T}$\,, the \emph{criterion at time $\theta$}  
for a strategy $(\tau, \sigma)$ $\in $ $\mathcal{T}_\theta^2$ is defined by
\begin{equation}\label{criterion}
I_{\theta} \left( \tau, \sigma \right) \coloneqq E \left[ \xi ( \tau ) {\bf 1}_{\left\{ \tau \leq \sigma \right\}} + 
\zeta ( \sigma ) {\bf 1}_{\left\{ \sigma < \tau \right\}} \given \mathcal{F}_{\theta} \right],
\end{equation}
the \emph{first}  or \emph{lower} value function at time $\theta$ is given by
\begin{equation}
\underbar{V} ( \theta ) \coloneqq \adjustlimits \esssup_{\tau \in \mathcal{T}_{\theta}} \essinf_{\sigma \in \mathcal{T}_{\theta}} I_{\theta} \left( \tau, \sigma \right),\label{vdessus}
\end{equation}
and, the \emph{second} or \emph{upper} value function at time $\theta$ is given by
\begin{equation}
\bar{V} ( \theta ) \coloneqq \adjustlimits \essinf_{\sigma \in \mathcal{T}_{\theta}} \esssup_{\tau \in \mathcal{T}_{\theta}} I_{\theta} \left( \tau, \sigma \right). \label{vdessous}
\end{equation}
For each $\theta \in \mathcal{T}$\,,  inequality $\underbar{V} ( \theta ) \leq \bar{V} ( \theta )$ a.s. clearly holds. \vspace{2pt}
\\
The game is considered to be \emph{fair} if $\underbar{V} ( \theta ) = \bar{V} ( \theta )$ a.s. and this quantity is then refered as the \emph{common} 
value function, or the \emph{value of the game} at time $\theta$.

We now introduce the following definition.
\begin{definition}
Let $\theta \in \mathcal{T}$. A pair $\left( \hat{\tau}, \hat{\sigma} \right) \in \mathcal{T}_{\theta}^2$
 is called a \emph{$\theta$-saddle point} if, for each $\left( \tau, \sigma \right) \in \mathcal{T}_{\theta}^2$:
\begin{equation}\label{saddlepoint}
I_{\theta} \left( \tau, \hat{\sigma}) \leq I_{\theta} \left( \hat{\tau}, \hat{\sigma} \right) \leq I_{\theta} ( \hat{\tau}, \sigma \right) \quad \mbox{a.s.}
\end{equation}
\end{definition}

In the study of the Dynkin game problem, the first aim is to provide some sufficient conditions under which 
the game is fair and, in this case, to characterize the common value function. The second aim is to address the question of the existence of saddle points.

\begin{remark}
By classical results on game problems, for each $\theta \in \mathcal{T}$\,, a pair $\left( \hat{\tau}, \hat{\sigma} \right)$ is 
a $\theta$-saddle point if and only if $\underbar{V} ( \theta ) = \bar{V} ( \theta )$ a.s. and the essential infimum in~\eqref{vdessous} 
and the essential supremum in~\eqref{vdessus} are respectively attained at $\hat{\sigma}$ and $\hat{\tau}$.
Hence, if, at initial time $\theta$, $\left( \hat{\tau}, \hat{\sigma} \right)$ is a $\theta$-saddle point, then $\hat{\tau}$ is an optimal 
strategy for the first player and $\hat \sigma$ is an optimal strategy for the second one.
\end{remark}

\begin{remark}\label{re} One could ask if the condition $\xi(T)=\zeta(T)=0$ is restrictive?\\
First, as  the criterion does not depend on the terminal reward $\zeta (T)$,  the assumption $\xi ( T ) = \zeta (T)$ is clearly not restrictive. 

Second, the additional assumption that $\xi ( T ) = \zeta (T)=0$ is no more restrictive. Let us show this assertion.
Let $\xi$ and $\zeta$ be two general integrable families such that $\xi ( T )$ $= \zeta (T)$, but not necessarily equal to $0$. 
Let us define the integrable families $\xi'$ and $\zeta'$ in $\mathcal{A}$ by 
 \[ \forall \theta \in \mathcal{T}, \quad \xi' ( \theta ) \coloneqq  \xi ( \theta ) - E \left[ \xi ( T ) \given 
 \mathcal{F}_{\theta} \right] \quad \text{and} \quad \zeta' ( \theta ) \coloneqq  \zeta ( \theta ) - E \left[ \xi ( T ) \given \mathcal{F}_{\theta} \right]. \]
We clearly have $\xi' \left( T \right) = \zeta' \left( T \right) = 0$ a.s. For every $\theta \in \mathcal{T}$ and every 
$\left( \tau, \sigma \right) \in \mathcal{T}_{\theta}^2$, the criterion $I_{\theta} \left( \tau, \sigma \right)$ associated to $\xi$ and $\zeta$ can be written:
\[ I_{\theta} \left( \tau, \sigma \right) = E \left[ \xi' \left( \tau \right) {\bf 1}_{\left\{ \tau \leq \sigma \right\}} + 
\zeta' \left( \sigma \right) {\bf 1}_{\left\{ \sigma < \tau \right\}} \given \mathcal{F}_{\theta} \right] +  E \left[ \xi ( T ) \given \mathcal{F}_{\theta} \right]. \]
As $E \left[ \xi ( T ) \given \mathcal{F}_{\theta} \right]$ does not depend of the strategies $(\tau,\sigma)$, 
solving the game problem associated with payoffs $\xi$ and $\zeta$ reduces to solving the game associated to the integrable payoffs families $\xi'$ and $\zeta'$.  

Of course, in order to have $\xi' \in \mathcal{S}^-$ and $\zeta'\in \mathcal{S}^+$, we must suppose that $\xi$, $\zeta$
 satisfy
\[ E \left[ \esssup_{\theta \in \mathcal{T}} \Big(  \xi ( \theta ) - E \left[ \xi ( T ) \given \mathcal{F}_{\theta} \right] \Big)^- \right]
 < +\infty \quad \text{and} \quad E \left[ \esssup_{\theta \in \mathcal{T}} \Big(  \zeta ( \theta ) -E \left[ \xi ( T ) \given \mathcal{F}_{\theta} \right]\Big)^+ \right] < +\infty.\]
Note that this condition is satisfied when $\xi \in \mathcal{S}^-$, $\zeta\in \mathcal{S}^+$ and $\xi (T) \in L^p$ with $p>1$, which is a 
classical assumption made in the literature (see \cite{CK}, among others).

\end{remark}

\section{Preliminary results}
In this section, we first provide the construction of two $[0, + \infty]$-valued supermartingale families $J$ and $J'$ such that 
$J=\mathcal{R}(J'+\xi)$ and $J'= \mathcal{R}(J-\zeta)$. In the case of processes, this construction is classical (see for example \cite{A}) 
and is done under Mokobodski's condition on $\xi$ and $\zeta$, which  stipulates that there exists two a.s. finite nonnegative supermartingale
 families $H$ and $H'$ such that  $\xi\le H-H'\le \zeta$.
In the present work, we do not need any condition of this type in order to define $J$ and $J'$. 
%Mokobodski's condition then appears to be equivalent to the fact that $J(0)$ or $J'(0)$ is finite. 
We show that when  $J$ is a.s. finite, or equivalently,   $J'$ is a.s. finite,   $J-J'$ is a well defined admissible family that satisfies $\xi\le J-J'\le \zeta$.
\subsection*{Construction of \boldmath$ J $ and $ J'$\unboldmath}
For each $ \theta \in \T$\,, set
\[ J_0 ( \theta ) \coloneqq 0 \quad \text{and} \quad J'_0 ( \theta ) \coloneqq 0 \]
and, let us introduce for each $n \in \N$\,,
\begin{align}
J_{n+1} ( \theta ) &\coloneqq \esssup_{\tau \in \T_{\theta}} E \left[ J'_n \left( \tau \right) + \xi ( \tau ) \given \F_{\theta} \right], \label{znplus} 
\\
J'_{n+1}( \theta ) &\coloneqq \esssup_{\sigma \in \T_{\theta}} E \left[ J_n \left( \sigma \right) - \zeta ( \sigma ) \given \F_{\theta} \right], \label{znmoins}
\end{align}
which are well defined by the following lemma.
\begin{lemma}\label{positif}
For each $n \in \N$ and each $\theta \in \T$\,, the random variables $J_n ( \theta )$ and $J'_n ( \theta )$ are well defined 
and nonnegative, namely $[0, + \infty]$-valued. Moreover, the families $J_n= \left( J_n(\theta), \,\theta \in \T \right)$ and 
$J'_n= \left( J'_n(\theta), \,\theta \in \T \right)$ are admissible and satisfy
\[ J_{n+1} = {\cal R} (J'_n + \xi) \quad \text{and} \quad J'_{n+1} = {\cal R} (J_n - \zeta). \]
In other words, the families $J_{n+1}$ and $J'_{n+1}$ are the smallest supermartingale families greater (almost surely) than $ J'_n + \xi $ and respectively $ J_n -\zeta$.
\end{lemma}
\begin{proof}{}
Let us show this property by induction. First, it clearly holds for $J_0$ and $J_0'$ since they are equal to $0$. Let us suppose 
that for a fixed $n \in \N$, for each $\theta$ $\in$ $\T$\,, $J_n(\theta)$ and $J'_n(\theta)$ are well-defined $\bar{\R}^+$-valued random variables. 
We then show that this property still holds for $n + 1$. Let $\theta \in \T$. Since, for every $\tau \in \T_{\theta}$, $J'_n ( \tau ) \geq 0$ a.s.\,and, since by $\xi \in \S^-$, that is 
\[ E \left[ \esssup_{\theta \in \T} \left(  \xi ( \theta ) ^- \right)\right]  < +\infty, \]
the random variable $J_{n+1} ( \theta )$ is well defined by (\ref{znplus}). Using \eqref{znplus}, the induction
 hypothesis (non-negativity of $J'_n$) and the equality $\xi ( T ) = 0 $ a.s., we derive that
\[ J_{n+1} ( \theta ) \geq E \left[ J'_n \left( T \right) + \xi ( T ) \given \F_{\theta} \right] \geq 0 \quad \mbox{a.s.} \]
By similar arguments and since $\zeta \in \S^+$,
we also have that $J'_{n+1}( \theta )$ is well defined and nonnegative. Moreover, by Proposition \ref{prop.snell}, 
$J_{n+1}$ and $J'_{n+1}$ are the Snell envelopes of $ J'_n + \xi $ and $ J_n -\zeta$ respectively. The proof is thus complete.
\end{proof}

\begin{lemma}\label{lem.JJn}
The sequences of families $\left( J_n \right)_{n \in \N}$ and $\left( J'_n \right)_{n \in \N}$ are non decreasing sequences of nonnegative supermartingale families.
\end{lemma}
\begin{proof}{}
The property can be proven by induction. By the previous lemma, we have
$J_1 \geq 0= J_0$ and $J'_1 \geq 0=J'_0$. Let us now suppose that, for a fixed $n \in \N^*(= \N \setminus \{0\})$, we have $J_n \geq J_{n - 1}$ and $J'_n \geq J'_{n - 1}$. We then have 
\[\mathcal{R} \left( J'_n + \xi \right) \geq \mathcal{R} \left( J'_{n - 1} + \xi \right) \quad {\rm and} \quad
\mathcal{R} \left( J_n - \zeta \right) \geq \mathcal{R} \left( J_{n - 1} - \zeta \right),\]
which leads to $J_{n+1} \geq J_n$ and $J'_{n+1}\geq J'_n.$
This concludes the proof.
\end{proof}
For each $\theta \in \T$\,, let us define $ J  ( \theta ) := \displaystyle\limsup_{n \to +\infty} J_n ( \theta )$ and 
$ \displaystyle J' ( \theta ) = \limsup_{n \to +\infty} J'_n ( \theta )$.
We clearly have:
\[
 J  ( \theta ) = \ulim_{n \to +\infty} J_n ( \theta ) \quad \mbox{a.s.} \quad {\rm and} \quad J' ( \theta ) = \ulim_{n \to +\infty} J'_n ( \theta ) \quad \mbox{a.s.}
\]
Note that the families $ J  = \left(  J  ( \theta ), \,\theta \in \T \right)$ and $ J' = \left(  J' ( \theta ),\,\theta \in \T \right)$ are clearly 
admissible and nonnegative (since, for each $n \in \N$\,, $J_n$ and $J'_n$ are themselves nonnegative).
\begin{theorem}\label{zpluszmoins} 
The families $ J $ and $ J'$ are nonnegative supermartingale families which satisfy
 $J = \mathcal{R} \left( J' + \xi \right)$ and $J' = \mathcal{R} \left( J - \zeta \right)$ that is, for each $\theta \in \T$\,,
\begin{align}
  J  ( \theta ) &= \esssup_{\tau \in \T_{\theta}} E \left[  J' \left( \tau \right) + \xi ( \tau ) \given \F_{\theta} \right] \quad \mbox{a.s.}, \label{zplus} \\
 J' ( \theta ) &= \esssup_{\sigma \in \T_{\theta}} E \left[  J  \left( \sigma \right) - \zeta ( \sigma ) \given \F_{\theta} \right] \quad \mbox{a.s.} \label{zmoins}
\end{align}
Moreover, $ J $ and $ J'$ are minimal in the following sense: if $\bar J$ and $\bar J'$ are two nonnegative supermartingale families satisfying the above system, that is, $\bar J = \mathcal{R} \left( \bar J' + \xi \right)$ and $\bar J' = \mathcal{R} \left( \bar J - \zeta \right)$,
then we have $ J  \leq \bar J$ and $ J' \leq \bar J'$.
\end{theorem}
Before giving the proof of the above theorem, we first show that the limit of a non decreasing sequence of nonnegative supermartingale families is a supermartingale family. More precisely, the following lemma holds.
\begin{lemma}\label{limitsuper}
Suppose that $(\phi_n)_{n \in \N}$ is a non decreasing sequence of nonnegative supermartingale families.
The family $\phi = \big(\phi(\theta),\theta\in \T\big)$ defined by $\displaystyle \phi(\theta)= \lim_{n \to \infty} \uparrow \phi_n(\theta)$ a.s.\,,  for each $\theta$ $\in$ $\T$\,, is then a nonnegative supermartingale family.
\end{lemma}
Note that $\displaystyle \phi (\theta)(\omega)$ is here only defined for $P$-almost all $\omega \in \Omega$. Classicaly, to define it for all $\omega$, it suffices to set $\displaystyle \phi(\theta)(\omega):= \limsup_{n \to \infty} \phi_n(\theta)(\omega)$.
\begin{proof}{}
One can easily prove that $\phi$ is an admissible family. Let us prove that it is a supermartingale family.
Let $\theta$, $\theta^{'}$ $\in$ $\T$ be such that $\theta \leq \theta^{'}$ a.s.\,
By the monotone convergence theorem for the conditional expectation, we get
$$E[ \phi (\theta^{'}) \,|\,\F_{\theta}]= \lim_{n \to \infty} \uparrow E[\phi_n (\theta^{'}) \,|\,\F_{\theta}] \leq \lim_{n \to \infty} \uparrow \phi_n (\theta) = \phi(\theta)\,\,\, {\rm a.s.},$$
where the inequality follows from the supermartingale family property of $\phi_n$, for each $n$.
\end{proof}
\begin{proof}[Proof of Theorem \ref{zpluszmoins}]
By Lemma \ref{lem.JJn} and  Lemma~\ref{limitsuper},  $J$ and $J'$ are nonnegative supermartingale families. Moreover as  the Snell enveloppe operator $\mathcal{R}$ is nondecreasing (see Proposition \ref{pro.snell}), for each $n \in \N$\,, we have $J_{n+1} = \mathcal{R} \left( J'_n + \xi \right) \leq \mathcal{R} \left(  J' + \xi \right)$.
By letting $n$ tend to $+\infty$, we get that
\begin{equation} \label{unu}
 J  \leq \mathcal{R} \left(  J' + \xi \right).
\end{equation}
Now, for each $n \in \N$\,, $J_{n+1} = \mathcal{R} (J'_n + \xi)$, hence by Proposition \ref{pro.snell}, $J_{n+1} \ge J'_n + \xi$.
%\[ J_{n+1} \geq J'_n + \xi. \]
By letting $n$ tend to $+\infty$, we derive that
$  J  \geq  J' + \xi$. 
By the supermartingale property of $ J $ and the characterization of $\mathcal{R} \left(  J' + \xi \right)$ as the smallest supermartingale greater than $ J' + \xi$, it follows that $J  \geq \mathcal{R} \left(  J' + \xi \right)$.
This with \eqref{unu} yields that $J  = \mathcal{R} \left(  J' + \xi \right)$. By similar arguments, one can easily derive that $J' = \mathcal{R} \left( J - \zeta \right)$.

It remains to show the second assertion. Let $\bar J$ and $\bar J'$ be two nonnegative supermartingale families satisfying the above system, that is, $\bar J = \mathcal{R} \left( \bar J' + \xi \right)$ and $\bar J' = \mathcal{R} \left( \bar J - \zeta \right)$.
Let us first show by induction that, for each $n \in \N$\,, the following property
\begin{equation}\label{JJprime}
J_n \leq \bar J \quad {\rm and} \quad J'_n \leq \bar J'
\end{equation}
holds. First, we have $J_0 =0 \leq \bar J$ and $J'_0 =0 \leq \bar J'$. Suppose now that, for some fixed $n \in \N$, Property (\ref{JJprime}) holds at rank $n$. Then, $J_n - \zeta \leq \bar J- \zeta$ and $J'_n + \xi  \leq \bar J' + \xi$. As $\mathcal{R}$ is a non decreasing operator (see Proposition \ref{prop.snell}), $ J'_{n+1}=\mathcal{R}(J_n - \zeta) \leq  \mathcal{R}(\bar J- \zeta)= \bar J'$ and $ J_{n+1}=\mathcal{R}(J'_n + \xi ) \leq  \mathcal{R}(\bar J' + \xi)= \bar J$. Thus, Property (\ref{JJprime}) holds at rank $n+1$.

By letting $n$ tend to $+\infty$ in (\ref{JJprime}), we get $ J  \leq \bar J$ and $ J' \leq \bar J'$, which ends the proof of Theorem~\ref{zpluszmoins}.
\end{proof}

Note now that since $ J  \geq  J' + \xi$ and $ J' \geq  J  - \zeta$, we have
\begin{proposition}\label{Y}
The condition $J(0) < +\infty $ is equivalent to the condition $  J'(0) < +\infty.$\\
Moreover, if $J(0) < +\infty$, the family of random variables $Y$ given by 
$$
Y \coloneqq  J  -  J',
$$
is then well defined and satisfies
\begin{equation}\label{encadrement}
 \xi  \leq Y \leq  \zeta. 
 \end{equation}
\end{proposition}
Note that this proposition ensures that, if $J(0) < +\infty $, then  $\xi   \leq  \zeta$. In other words, if there exists $\nu \in \T$ such that $P \left(  \xi (\nu) > \zeta (\nu) \right) >0$, then $J(0)= J'(0) = + \infty$.
 \begin{remark} Note that the existence of two nonnegative a.s. finite supermartingales $H$ and $H'$ such that $\xi \le H-H'\le \zeta$ is known as Mokobodski's condition. Thus Proposition \ref{Y} shows that condition $J(0)<+\infty$ implies that Mokobodski's condition holds. In section \ref{section.moko}, we prove that these conditions are actually equivalent.
\end{remark}

\section{Existence and characterization of the value of the game}
When $J(0) < +\infty $, the family $Y=J-J'$ appears as a quite natural candidate to be the value of the game.

\subsection{Case when there exists optimal stopping times for $J$ and $J'$}

We provide a first result  involving $Y=J-J'$ as the common value function.
\begin{proposition}\label{critereselle}
Suppose that $J(0) < +\infty$. Let $\theta \in \T$ and let $\left( \hat{\tau}, \hat{\sigma} \right) \in \T_{\theta}^2$ be such that $\hat{\tau}$ is optimal for $ J  ( \theta )$, that is
\begin{align*}
 J  ( \theta ) &= \esssup_{\tau \in \T_{\theta}} E \left[  J' \left( \tau \right) + \xi ( \tau ) \given \F_{\theta} \right] = E \left[  J' ( \hat{\tau}) + \xi ( \hat{\tau} ) \given \F_{\theta} \right] \quad \mbox{a.s.},
\intertext{and $\hat{\sigma}$ is optimal for $ J' ( \theta )$, that is }  
 J' ( \theta ) &= \esssup_{\sigma \in \T_{\theta}} E \left[  J  \left( \sigma \right) - \zeta ( \sigma ) \given \F_{\theta} \right] = E \left[  J  ( \hat{\sigma} ) - \zeta  ( \hat{\sigma} ) \given \F_{\theta} \right] \quad \mbox{a.s.}.
\end{align*}
Then, the game is fair, the common value function is equal to $Y(\theta)(= J(\theta)-J'(\theta))$ and $\left( \hat{\tau}, \hat{\sigma} \right)$ is a $\theta$-saddle point. 
We thus have 
\begin{equation} \label{sa}
Y ( \theta ) = \underbar{V} ( \theta ) = \bar{V} ( \theta ) = I_{\theta} \left( \hat{\tau}, \hat{\sigma} \right) \quad \mbox{a.s.}
\end{equation}
\end{proposition}

\begin{example}
Note that if $\xi$ is a supermartingale family, then for each $\theta$ $\in$ $\T$\,, $ (\theta, T)$ is a $\theta$-saddle point. Indeed, for each 
$(\tau, \sigma)$ $\in$ $\T_{\theta}^2$\,, 
\[I_{\theta} (\theta, T)= I_{\theta} (\theta, \sigma)= E[ \xi (\theta)  \given \F_{\theta} ] = \xi(\theta)\quad \mbox{a.s.}\]
and $I_{\theta} (\tau, T)=  E[ \xi (\tau)  \given \F_{\theta} ] \leq \xi (\theta)\quad a.s.$ and the common value function is equal to $\xi$
(without any condition on $\zeta$). Hence, if additionally, there exists $\nu \in \T$ such that $P \left(  \xi (\nu) > \zeta (\nu) \right) >0$, then $J(0)= + \infty$ (by (\ref{encadrement})) even if for each $\theta$\,, a $\theta$-saddle point does exist. 

This example shows that  condition $\xi \leq \zeta$ is not a necessary condition to have the existence of saddle points.  In particular Mokobodski's condition is not necessary.

\end{example}

\begin{proof}{}
Let $\theta \in \T$. By the optimality criterion (see  Proposition \ref{prop.opt}), $\left( J(\tau), \,  \tau \in \T_{\left[ \theta, \hat{\tau} \right]}\right)$ is a martingale family and
$ J ( \hat{\tau}) =  J' ( \hat{\tau}) + \xi ( \hat{\tau} )$ a.s.\,, that is $Y ( \hat{\tau} ) = \xi ( \hat{\tau} )$ a.s.\, Also, $\left( J'(\sigma), \,  \sigma \in \T_{\left[ \theta, \hat{\sigma} \right]}\right)$ is a martingale family and $J'  ( \hat{\sigma} ) =  J  ( \hat{\sigma} ) - \zeta  ( \hat{\sigma} )$ a.s.\,that is $Y ( \hat{\sigma} ) = \zeta  ( \hat{\sigma} )$ a.s.\,
Since $Y =  J  -  J'$, it follows that $\left( Y(\alpha), \,  \alpha \in \T_{\left[ \theta, \hat{\tau} \wedge \hat{\sigma} \right]}\right)$ is a martingale family  and hence that
\begin{align*}
Y ( \theta ) &= E \left[ Y \left( \hat{\tau} \wedge \hat{\sigma} \right) \given \F_{\theta} \right] = E \left[ Y ( \hat{\tau} ) {\bf 1}_{\left\{ \hat{\tau} \leq \hat{\sigma} \right\}} + Y ( \hat{\sigma} ) {\bf 1}_{\left\{ \hat{\sigma} < \hat{\tau} \right\}} \given \F_{\theta} \right] \quad \mbox{a.s.}\\
&= E \left[ \xi ( \hat{\tau} ) {\bf 1}_{\left\{ \hat{\tau} \leq \hat{\sigma} \right\}} + \zeta  ( \hat{\sigma} ) {\bf 1}_{\left\{ \hat{\sigma} < \hat{\tau} \right\}} \given \F_{\theta} \right] = I_{\theta} \left( \hat{\tau}, \hat{\sigma} \right) \quad \mbox{a.s.}.
\end{align*}
Let us now show that, for each $\sigma \in \T_{\theta}$\,, $Y ( \theta ) \leq I_{\theta} \left( \hat{\tau}, \sigma \right)$ a.s.. Let $\sigma \in \T_{\theta}$. Since $\left( J(\tau), \,  \tau \in \T_{\left[ \theta, \hat{\tau} \right]}\right)$ is a martingale family and $ J'$ is a supermartingale family, it follows that $Y =  J  -  J'$ is a submartingale family on $\left[ \theta, \hat{\tau} \wedge \sigma \right]$, which ensures that:\[ Y ( \theta ) \leq E \left[ Y \left( \hat{\tau} \wedge \sigma \right) \given \F_{\theta} \right] \leq E \left[ \xi ( \hat{\tau} ) {\bf 1}_{\left\{ \hat{\tau} \leq \sigma \right\}} + \zeta ( \sigma ) {\bf 1}_{\left\{ \sigma < \hat{\tau} \right\}} \given \F_{\theta} \right] = I_{\theta} \left( \hat{\tau}, \sigma \right)\quad \mbox{a.s.}, \]
where the second inequality follows from the fact that $Y ( \hat{\tau} ) = \xi ( \hat{\tau} )$ a.s. and $Y \left( \sigma \right) \leq \zeta ( \sigma )$ a.s.\,By similar arguments, one can show that, for each $\tau \in \T_{\theta}$\,, $I_{\theta} \left(  \tau, \hat{\sigma} \right) \leq Y ( \theta ) $ a.s.\, We have thus proved that $\left( \hat{\tau}, \hat{\sigma} \right)$ is a $\theta$-saddle point for the game and that equalities (\ref{sa}) hold.
\end{proof}

\begin{remark}
Let $\bar J$ and $\bar J'$ be two nonnegative supermartingale families in $\S^-$, and such that $\bar J = \mathcal{R} \left( \bar J' + \xi \right)$ and $\bar J' = \mathcal{R} \left( \bar J - \zeta \right)$. The same proof shows that the above property still holds for $\bar J$ and $\bar J'$.\\
\indent
More precisely, if $\bar J(0) < +\infty$ and if, for some $\theta \in \T$ and $\left( \hat{\tau}, \hat{\sigma} \right) \in \T_{\theta}^2$, $\hat{\tau}$ is optimal for $\bar J ( \theta )$ and $\hat{\sigma}$ is optimal for $\bar J' ( \theta )$, then $\left( \hat{\tau}, \hat{\sigma} \right)$ is a $\theta$-saddle point and we have:
\begin{equation*}
\bar J ( \theta ) - \bar J' ( \theta ) = Y ( \theta ) = \underbar{V} ( \theta ) = \bar{V} ( \theta ) \quad \mbox{a.s.}
\end{equation*}
In particular, we have $\bar J ( \theta ) - \bar J' ( \theta ) = J ( \theta ) -  J' ( \theta )$ a.s.\,
%Similarly, in the sequel, the stated results still hold with $ J $ and $ J'$ replaced by $\bar J$ and $\bar J'$, respectively.
\end{remark}
We will see in the next section that it is not necessary to suppose the existence of saddle points so that the game is fair.

\subsection{Existence and characterization of the value of the game under right regularity assumptions on $\xi$ and $\zeta$}
Let us now introduce the following definition. 
\begin{definition}
A family $\phi=\left( \phi ( \theta ), \theta \in \T  \right)$ $\in$ $\mathcal{S}^-$ is said to be
 \emph{right-(resp. left-)upper semicontinuous in expectation along stopping times (right- (resp. left-) USCE)} if for 
 all $\theta \in \T$ and for all sequences of stopping times $ (\theta_n)_{n \in \N}$ such that $\theta_n\downarrow \theta$ 
(resp. $ \theta_n \uparrow \theta$) 
\begin{equation*}
E[\phi(\theta) ]\geq \limsup_{n\to \infty} E[\phi(\theta_n)].
\end{equation*}  
Moreover, $\phi$ is said to be USCE if it is both right and left-USCE.
\end{definition}

\begin{remark} \label{rusce} 
A non negative supermartingale family is right-USCE.
\end{remark}

\begin{theorem}\label{saun}
Suppose that $ J (0)< + \infty$ and that the families $(\xi(\theta),\theta \in\T)$ and $(-\zeta(\theta), \theta \in \T)$ are right-USCE.
Then, the game is fair and the common value function is equal to $Y(=J-J')$ that is, for each $\theta \in \T$\,,
\begin{equation}\label{equ}
Y(\theta) = \underline V(\theta) = \overline V(\theta)\,\,\, {\rm a.s.}\,
\end{equation}
\end{theorem}

\begin{proof}{} Let $\theta \in \T$. For each $\lambda \in ]0,1[$\,, we introduce 
\begin{equation*} \tau^{\lambda} (\theta):=
\essinf \{\,\tau \in \T_{\theta}\, , \,\lambda  J (\tau) \leq  J'(\tau)+ \xi(\tau) \,\,\mbox {a.s.} \,\}.
\end{equation*}
and
\begin{equation*}\sigma^{\lambda} (\theta):=
\essinf \{\,\sigma \in \T_{\theta}\, , \,\lambda  J'(\sigma) \leq  J (\sigma)- \zeta(\sigma) \,\,\mbox {a.s.} \,\}.
\end{equation*}

By Theorem \ref{thm.epsopt}, $\tau^{\lambda} (\theta)$ (resp. $\sigma^{\lambda} (\theta)$) is $(1 - \lambda)$-optimal for $ J (\theta)$ (resp. $ J'(\theta)$).
In order to simplify notation, in the sequel, $\tau^{\lambda} (\theta)$ (resp. $\sigma^{\lambda} (\theta)$) will be denoted by $\tau^{\lambda}$ (resp. $\sigma^{\lambda}$). Now, the following lemma holds.

\begin{lemma}\label{sadeuxbis}
For each $\lambda$ $\in$ $]0,1[$ and each $(\sigma, \tau)$ $\in$ $\T_{\theta}^2$, we have 
\begin{equation*}
I_{\theta}(\tau,  \sigma^{\lambda}) - (1 - \lambda) J'(\theta)\,\,\leq \,\, Y(\theta) \,\,\leq \,\,I_{\theta}( \tau^{\lambda}, \sigma) + (1 - \lambda)  J (\theta) \,\,\mbox {a.s.} 
\end{equation*}
\end{lemma}
We postpone for a while the proof of this lemma and complete the proof of  Theorem \ref{saun}.
\\
We clearly have that $ \underline V(\theta)\leq \overline V(\theta)$ a.s.\,Hence, it is sufficient to prove that 
\begin{equation}\label{satrois}
 \overline V(\theta) \,\,\leq \,\, Y(\theta) \,\,\leq \,\,\underline V(\theta) \,\,\, {\rm a.s.}\,
 \end{equation} 
Now, the previous lemma yields that  for each $\lambda$ $\in$ $]0,1[$\,,
\begin{equation*}
\esssup_{\tau \in \T_{\theta}} I_{\theta}(\tau,  \sigma^{\lambda}) - (1 - \lambda) J'(\theta)\,\,\leq \,\, Y(\theta) \,\,\leq 
\,\,\essinf_{\sigma \in \T_{\theta}}I_{\theta}( \tau^{\lambda}, \sigma) + (1 - \lambda)  J (\theta) \,\,\mbox {a.s.}\,,
\end{equation*}
which implies that
%\begin{equation*}
%\essinf_{\sigma \in \T_{\theta}}\esssup_{\tau \in \T_{\theta}} I_{\theta}(\hat \sigma, \tau) - (1 - \lambda) J'(\theta)\leq  Y(\theta) \leq \esssup_{\tau \in T_S}\essinf_{\sigma \in T_S}I_{\theta}(\sigma, \tau) + (1 - \lambda)  J (\theta) \,\,\mbox {a.s.}\,
%\end{equation*}
\begin{equation*}
 \overline V(\theta) - (1 - \lambda) J'(\theta)\,\,\leq \,\, Y(\theta) \,\,\leq \,\, \underline V(\theta) + (1 - \lambda)  J (\theta)\,\,\, {\rm a.s.}\,
 \end{equation*}
By letting $\lambda$ tend to $1$, we get inequalities (\ref{satrois}).
It follows that $ \overline V(\theta) = Y(\theta) =\underline V(\theta)$ a.s.\,, and completes the proof of Theorem \ref{saun}.\end{proof}

It remains to prove Lemma \ref{sadeuxbis} which actually shows that $(\sigma^{\lambda}, \tau^{\lambda})$ is an $(1- \lambda)$-saddle point.
\begin{proof}[Proof of Lemma \ref{sadeuxbis}] First, one can easily show that each supermatingale family 
is right-USCE. Hence, since $ J $ and $ J'$ are supermartingale families, they are right-USCE and so are $ J  -\zeta$ and $ J' + \xi$ because $\xi$ and $-\zeta$ are right-USCE.\\
By Remark \ref{rusce}, since $ J $ and $ J'$ are supermartingale families, they are right-USCE and so are $ J  -\zeta$ and $ J' + \xi$ because $\xi$ and $-\zeta$ are right-USCE.\\
Now, by Theorem \ref{thm.epsopt},
$\big(\,  J'(\sigma), \sigma \in \T_{[\theta,  \sigma^{\lambda}]} \,\big)$ is a martingale family.
 Hence, since $ J $ is a supermartingale family, it follows that $\big(\,  Y(\sigma), \sigma \in \T_{[\theta,  \sigma^{\lambda}]} \,\big)$ 
 is a supermartingale family, because $Y=  J -  J'$. We thus have
\begin{equation}\label{firsta}
Y(\theta)\geq E[ Y( \sigma^{\lambda}\wedge \tau)\,|\,\F_{\theta}]= E[Y( \tau) {{\bf 1}}_{\{ \tau\leq   \sigma^{\lambda}\} }+
 Y( \sigma^{\lambda}) {{\bf 1}}_{  \{\sigma^{\lambda} < \tau \}} \,|\,\F_{\theta}]\,\,\, {\rm a.s.}\,
\end{equation}
Recall now that $Y\geq \xi$. Moreover, thanks to inequality (\ref{lambdai}) in the Appendix, we have the following inequality
$\lambda  J'( \sigma^{\lambda})\leq  J ( \sigma^{\lambda})- \zeta( \sigma^{\lambda})$ a.s.\,, which can be written
\[Y( \sigma^{\lambda}) \geq \zeta( \sigma^{\lambda})- (1- \lambda)  J'( \sigma^{\lambda})\,\,\, {\rm a.s.}\]
This with inequality (\ref{firsta}) leads to 
\[Y(\theta) \geq E[\xi( \tau) {{\bf 1}}_{\{ \tau \leq   \sigma^{\lambda} \}}+\zeta( \sigma^{\lambda}) {{\bf 1}}_{ \{ \sigma^{\lambda} < \tau \}}\,|\,\F_{\theta}] - (
1- \lambda) E[  J'( \sigma^{\lambda}) {{\bf 1}}_{ \{ \sigma^{\lambda} < \tau\}}\,|\,\F_{\theta}]   \,\,\, {\rm a.s.}\,\]
The supermatingale property of $ J'$ yields that
\[E[  J'( \sigma^{\lambda}) {{\bf 1}}_{  \{\sigma^{\lambda} < \tau\}}\,|\,\F_{\theta}] \leq 
E[  J'( \sigma^{\lambda}) \,|\,\F_{\theta}]  \leq   J'(\theta)
  \,\,\, {\rm a.s.}\,\]
This with the previous inequality and the definition of $I_{\theta}(\tau,  \sigma^{\lambda})$ leads to
\[Y(\theta) \geq I_{\theta}(\tau,  \sigma^{\lambda}) - (1- \lambda)  J'(\theta)  \,\,\, {\rm a.s.}\]
By the same arguments, one can show the following inequality:
\[Y(\theta) \,\,\leq \,\,I_{\theta}( \tau^{\lambda}, \sigma) + (1 - \lambda)  J (\theta) \,\,\mbox {a.s.}\,\]
The proof of Lemma \ref{sadeuxbis} is thus complete.
\end{proof}

\begin{remark}
Let $\bar J$, $\bar J'$ be two nonnegative supermartingale families such that $\bar J = \mathcal{R} \left( \bar J' + \xi \right)$ and $\bar J' = \mathcal{R} \left( \bar J - \zeta \right)$. The same proof shows that the above property still holds for $\bar J$ and $\bar J'$.\\
More precisely, if $\bar J(0) < +\infty$ and if $\xi$ and $-\zeta$ are right-USCE, then 
$Y(\theta) = \bar J(\theta) - \bar J' (\theta)$ a.s. and equalities (\ref{equ}) hold.
\end{remark}

\section{Existence of saddle points}
By Theorem \ref{zpluszmoins}, $J$  and $J'$  are the value functions associated with the optimal stopping problems with rewards $J'+\xi$  and  $J-\zeta$ respectively. 
 By Proposition \ref{critereselle}, if these two optimal stopping problems admit optimal stopping times,
 respectively $\hat \tau$ and $\hat \sigma$, the pair $(\hat \tau , \hat \sigma)$ is then a saddle point for the game problem.
Now, Theorem \ref{thm.exopt} ensures that, if a reward is USCE, then there exists an optimal stopping time for the associated optimal stopping problem. Thus, the natural question arises: under which conditions on $\xi$ and $\zeta$, the families $J' + \xi$ and $J- \zeta$ are USCE?

By Remark \ref{rusce}, since $ J $ and $ J'$ are supermartingale families, they
are right-USCE. Hence, if $\xi$ and $-\zeta$ are right-USCE, so are $ J  -\zeta$ and $ J' + \xi$.

We now introduce the following definition.
\begin{definition} A uniformly integrable family $( \phi(\theta), \theta \in \T)$ in $\S^-$ is said to be {\em strongly left-upper semicontinuous along stopping times 
in expectation} (strong left-USCE) if for all $\theta \in \T$, for all $F$ $\in$ ${\cal F}_{\theta^-}$, and for all non decreasing sequences of stopping times $ (\theta_n)_{n \in \N}$ such that $ \theta_n \uparrow \theta$, 
\begin{equation}\label{uscef}
 \limsup_{n\to \infty} E[\phi(\theta_n){{\bf 1}}_{F}] \leq E[\phi(\theta)  {{\bf 1}}_{F} ].
 \end{equation} 
 \end{definition}
 
 \begin{remark}Note that in this definition, no condition is required at a totally inaccessible stopping time (such as, for example, a jump of a Poisson process). Indeed, suppose that $\theta$ is a totally inaccessible stopping time. Then, if $ (\theta_n)_{n \in \N}$ is a non decreasing sequence of stopping times converging to $\theta$, it is necessarily a.s. constant equal to $\theta$ from a certain rank. Using the integrability conditions and Fatou's lemma, we get 
\[ \limsup_{n\to \infty} E[\phi(\theta_n){{\bf 1}}_{F}] \leq  E[\limsup_{n\to \infty} \phi(\theta_n){{\bf 1}}_{F}] = E[\phi(\theta)  {{\bf 1}}_{F} ].\]
Hence, inequality (\ref{uscef}) is always satisfied for any $F$ $\in$ ${\cal F}_{\theta^-}$.\\
\indent In the particular case of an optional process $(\phi_t)$, the strong left-USCE property of the family $\big( \phi_{\theta},   \theta \in \T \big)$ is thus
weaker than the usual left-upper semicontinuity property of the process $(\phi_t)$.
 \end{remark}

We  provide the following regularity result.
\begin{theorem}\label{left-USCE} 
Suppose the families 
$J$ and $J'$ are uniformly integrable. Suppose also that, for each predictable stopping time $\tau \in \T$, on  $\{\tau <T\}$, 
\begin{equation}
\label{SUun}
\{ \xi(\tau) = \zeta(\tau) \} = \emptyset \quad {\rm a.s.}
\end{equation}
and that $\xi$ and $\zeta$ are left-limited along stopping times at $T$ with $\{\xi(T^-) = \zeta(T^-)\} = \emptyset$ a.s.

%Suppose that for each $\theta \in \T_0$ and for each non decreasing sequence of stopping times $(\theta_n)_{n \in \N}$ such that $ \theta^n \uparrow \theta$, we have
%\begin{equation} \label{SU}
%\{ \xi(\theta) = \zeta(\theta) \} \cap \{\theta_n<\theta,\,\, \mbox{for all}\,\,n\, \}= \emptyset \quad {\rm a.s.}
%\end{equation}
If $\xi$ and $-\zeta$ are right-USCE and strong left-USCE, then  the families $J$ and $J'$ are USCE (that is right- and left-USCE).
\end{theorem}
This theorem together with Proposition \ref{critereselle} and the  existence of  optimal stopping time for USCE  
(see Theorem \ref{thm.exopt}) provides the following general existence result.
\begin{corollary}\label{existence}{(\bf \em Saddle-point existence result)}
Suppose that the assumptions of the above theorem hold.
For each $\theta \in \T$\,, the stopping time
\begin{equation}\label{tauetoile}
 \tau_{*} (\theta):=
\essinf \{\,\tau \in \T_{\theta}\, , \,  J (\tau) =  J'(\tau)+ \xi(\tau) \,\,\mbox {a.s.} \,\}.
\end{equation}
is an optimal stopping time for $ J  ( \theta )$ 
%(that is (\ref{Jun}) holds) 
and 
\begin{equation} \label{sigmaetoile}
\sigma_* (\theta):=
\essinf \{\,\sigma \in \T_{\theta}\, , \, J'(\sigma) =  J (\sigma)- \zeta(\sigma) \,\,\mbox {a.s.} \,\}.
\end{equation}
is an optimal stopping time for  $ J' ( \theta )$. 
%(that is (\ref{Jdeux}) holds). 
Moreover, the pair $\left(  \tau_{*} (\theta),\sigma_* (\theta) \right)$ is a $\theta$-saddle point for the criterion $I_{\theta}$ and
\[Y ( \theta ) = \underbar{V} ( \theta ) = \bar{V} ( \theta ) = I_{\theta} \left( \tau_{*} (\theta), \sigma_* (\theta)\right) \quad \mbox{a.s.}.\]
\end{corollary}

\begin{remark}
Let us consider the particular case where the families $\xi=\left( \xi ( \tau ), \tau \in \T  \right)$ and\\ $\zeta=\left( \zeta ( \tau ), \tau \in \T  \right)$ are defined via given predictable processes 
$(\xi'_t)$ and $(\zeta'_t)$ by $\xi(\tau):= \xi'_\tau$ and $\zeta(\tau):= \zeta'_\tau$. Then, by classical results on processes (see Dellacherie and Meyer (1977)), condition (\ref{SUun}) is equivalent to the fact that $P(\xi'_t < \zeta'_t,\,\, 0< t<T)=1$. Of course, this equivalence does not hold if $(\xi'_t)$ and $(\zeta'_t)$ are only supposed to be optional.\\
 Moreover, Theorem \ref{left-USCE} still holds under slightly different assumptions (see Proposition 
\ref{ajoute}).

Note also that the assumptions of the above existence result are weaker than those made by Alario-Nazaret, Lepeltier and Marchal in \cite{A} (see section 2 p.30). Moreover, their proof of the left-upper semicontinuity property of the processes $(J_t)$ and $(J'_t)$, where 
$(J_t)$ and $(J'_t)$ are the processes which aggregate the families $J$ and $J'$, requires highly sophisticated  results of the General Theory of Processes as the so called Mertens decomposition of supermartingales and the existence of a left-upper semicontinuous envelope process $(\underbar X_t)$
 for a given optional process $(X_t)$ (see Lemma 4-2 p.30 in \cite{A}). On the contrary, the proof given below is only based on classical properties of Probability Theory.
\end{remark}

Before showing Theorem \ref{left-USCE}, we provide the following lemma.
%\begin{remark} By using some results of Dellacherie and Meyer (1975) (IV. Th.81), Condition (\ref{SU}) can be shown to be equivalent to the fact that, for each predictable stopping time $\theta \in \T_0$, $$\{ \xi(\theta) = \zeta(\theta) \} = \emptyset \quad {\rm a.s.}$$
%(see Proposition \ref{dm} in the Appendix for details).
%Condition (\ref{SU}) can be shown to be equivalent to the fact that, for each $\theta \in \T_0$, $\{ \xi(\theta) = \zeta(\theta) \} \cap A_{\theta}= \emptyset$ a.s.\,, where $A_{\theta}$ denotes the set of atteignability of $\theta$. 
%Indeed,  for each $\theta \in \T_0$,
%by a result of Dellacherie and Meyer (Chap IV.80), there exists a sequence of sets $(A_k)_{k \in \N}$ in ${\cal F}_{\theta^-}$ such that  for each $k$, $\theta$ is accessible on $A_k$ and $A(\theta) = \cup_k A_k$ a.s.  
%\end{remark}

 \begin{lemma}\label{DM} Suppose that for each predictable stopping time $\tau\in \T$, on  $\{\tau <T\} $,
\[\{ \xi(\tau) = \zeta(\tau) \} = \emptyset \quad {\rm a.s.}\]
Then, for each $\theta \in \T$  and for each non decreasing sequence of stopping times $(\theta_n)_{n \in \N}$ such that $ \theta^n \uparrow \theta$, we have
\begin{equation*} 
\Big\{ \xi(\theta) = \zeta(\theta) \Big\} \cap \Big\{\theta_n<\theta<T,\,\, \mbox{for all}\,\,n\, \Big\}= \emptyset \quad {\rm a.s.}
\end{equation*}
  \end{lemma}

\begin{proof}[Proof of Lemma \ref{DM}] Let us introduce the set $A$ $:=$ $\Big\{\theta_n<\theta<T,\,\, \mbox{for all}\,\,n\, \Big\}$.\\ 
Let us first show that $\theta$ coincides on $A$ with a predictable stopping time. Let 
\[\tau_n := (\theta_n {\bf 1}_{ \{\theta_n < \theta \}}+ T{\bf 1}_{ \{\theta_n \geq \theta \}}) \wedge (T-\frac{1}{n}),\]
for each $n$.
The sequence $(\tau_n)_{n \in \N}$ announces its limit $\tau$ everywhere on $\Omega$. Hence, $\tau$ is predictable.
Also, $\tau = \theta$ a.s. on $A$. 
Hence, we get $$\{ \xi(\theta) = \zeta(\theta) \} \cap A \,= \,\{ \xi(\tau) = \zeta(\tau) \} \cap A \,\subset \, \{ \xi(\tau) = \zeta(\tau) \}\,=\, \emptyset \quad {\rm a.s.}\,,$$
which provides the desired result. 
\end{proof}

\begin{proof}[Proof of Theorem \ref{left-USCE}] Since $J$ and $J'$ are supermartingale families, one can easily prove that they are right-USCE. It thus remains to prove the left-USCE property.
Let $\theta$ be a given stopping time and let $(\theta_n)_{n \in \N}$ be a non decreasing sequence of stopping times such that 
$ \theta^n \uparrow \theta$. Let us show that $\limsup_{n \to + \infty} E[J(\theta_n)] \leq E[J(\theta)]$.

Let us define $A$ $:=$ $\{\theta_n<\theta,\,\, \mbox{for all}\,\,n\, \}$.
Since for almost every $\omega$ $\in$ $ A^c$, the sequence $(\theta_n(\omega))_{n \in \N}$ is constant from a certain rank, it follows that the sequence $(J(\theta_n)(\omega))_{n \in \N}$ is also constant from a certain rank. 
Indeed,  note first that $A^c := \cup_p \cap_{l \geq p} \{ \theta_l = \theta\}$. Let $p \in \N$. By the admissibility property of $J$, for each $n \geq p$\,, 
$J(\theta_n) = J(\theta)$ a.s.\,on $\cap_{l \geq p} \{ \theta_l = \theta\}$. Hence,
$\lim_{n\to\infty} J(\theta_n)= J(\theta)$ a.s.\,on $\cap_{l \geq p} \{ \theta_l = \theta\}$ and this holds for each $p$. Hence, $\lim_{n\to\infty} J(\theta_n)= J(\theta)$ a.s.\,on $ A^c$.

Since $(J(\theta_n))$ is uniformly integrable, we have $\lim_{n \to + \infty} E[J(\theta_n){{\bf 1}}_{A^c}] = E[J(\theta){{\bf 1}}_{A^c}]$. It is thus sufficient to show that 
\[\limsup_{n \to + \infty} E[J(\theta_n){{\bf 1}}_A] \leq E[J(\theta){{\bf 1}}_A].\]
Since $(\theta_n)$ announces $\theta$ on $A$, by the convergence theorem for nonnegative discrete supermartingales, 
the sequence $\left( J(\theta_n) \right)_{n \in \N}$ converges a.s. to a nonnegative random variable we denote by $J(\theta^-)$. 
Also, the random variable $J(\theta^-){\bf 1} _A$ is ${\cal F}_{\theta^-}$-measurable and, if $(\theta'_n)_{n \in \N}$ is a non decreasing sequence of stopping times such that $ \theta'_n \uparrow \theta$, then
 $\lim_{n\to\infty} J(\theta'_n)= \lim_{n\to\infty} J(\theta_n) = J(\theta^-)$ a.s.\,on $A \cap A'$, 
 where $A':=\{\theta'_n<\theta,\,\, \mbox{for all}\,\,n\, \}$, as precised in Lemma \ref{sull} in the Appendix.

Similarly, there exists a nonnegative random variable which we denote  $J'(\theta^-)$ such that $\lim_{n\to\infty} J'(\theta_n)= J'(\theta^-)$ a.s.\, This random $J'(\theta^-)$ satisfies similar properties as $J(\theta^-)$.
%Let $\xi$ be stopping time accessible on a set $A \subset \Omega$ and let $(S_n)$ $\in $ ${\cal A}(S,A)$. 
%It is clear that $(\zeta( S_n))_{ n\in \N}$ is a discrete nonnegative supermartingale relatively to the filtration $(\F_{S_n})_{n\in \N}$. By the well-known convergence theorem for discrete supermartingales, there exists a random variable $Z$ such that  
%$(\zeta( S_n))_{ n\in \N}$ converges a.s. to  $Z$. If $\zeta(0) < +\infty$, then $Z$ is integrable.
%Set $\zeta(S^-) := Z$. 
%It remains to show that this limit, on $A$, does not depend on the sequence $(S_n)$. 
%Let $(\theta'_n)_{n \in \N}$ be a non decreasing sequence of stopping times such that 
%$ \theta'_n \uparrow \theta$. Suppose that for each $n$ $\in$ $\N$, $\theta'_n < \theta$ a.s. on $A$.
%Again, by the supermartingales convergence theorem, there exists a  random variable $J'$ such that  
%$(\zeta(S'_n))_{ n\in \N}$ converges a.s. to  $J'$. We will now prove that $Z= J'$ a.s.\,on $A$. 
%
%Moreover, by Theorem 4.5 in Kobylanski and Quenez (2011), there exists an ${\cal F}_{\theta^-}$-measurable $Z({\theta^-})$ (resp. $J'({\theta^-})$) unique on $A$ up to the equality a.s. such that $\lim_{n\to\infty} J(\theta_n)= J(\theta^-)$ (resp. $\lim_{n\to\infty} %J'(\theta_n)= J'(\theta^-)$) a.s.\,on $A$.\\
%

Since $(J(\theta_n))$ is uniformly integrable, we have $\lim_{n \to + \infty} E[J(\theta_n){{\bf 1}}_A] 
= E[J(\theta^-){{\bf 1}}_A]$.
The problem thus reduces to prove that 
\begin{equation}\label{conclusion}
E[\left(J(\theta) - J(\theta^-)\right){{\bf 1}}_A] \geq 0.
\end{equation}

Let $B:= \left\{\,E[J(\theta) \, |\, {\cal F}_{\theta^-}] < J(\theta^-) \, \right\}\cap A$. 
Since $J$ is a supermartingale family, we have $E[J(\theta) \, |\, {\cal F}_{\theta^-}] \leq J(\theta^-)$ a.s.\,on  $A$. It follows that $E[J(\theta) \, |\, {\cal F}_{\theta^-}] = J(\theta^-)$ a.s.\,on  $A\setminus B$. We thus have
\begin{eqnarray*}
E\left[\left(J(\theta) - J(\theta^-)\right){{\bf 1}}_A\right]  &=& E\left[  \left( E[J(\theta) \, |\, {\cal F}_{\theta^-}] - J({\theta^-}) \right){{\bf 1}}_A \right]\\
&=& E\left[ \left( E[J(\theta) \, |\, {\cal F}_{\theta^-}] - J(\theta^-)\right) {{\bf 1}}_{ B}\right]\\
&=& E\left[ \left( J(\theta) - J(\theta^-)\right) {{\bf 1}}_{ B}  \right]
\end{eqnarray*}

For each $p $ and for each $\lambda \in ]0,1[$\,, let us define 
$$\tau^{\lambda} (\theta_p) :=
\essinf \{\,\tau \in \mathcal{T}_{\theta_p}\, , \,\lambda J(\tau) \leq J'(\tau)+ \xi(\tau) \,\,\mbox {a.s.} \,\}$$
By Lemma \ref{Propriete 2}, the sequence of stopping times $(\tau^{\lambda} (\theta_p))_{p \in \N}$ announces $\theta$ on $B$.

\noindent Moreover, by Lemmas \ref{Propriete 3} and \ref{sull}, and since $J'$ is uniformly integrable,
\begin{eqnarray*}\label{premier}
E[ J(\theta^-){{\bf 1}}_{B}] &=&   \sup_{\lambda \in ]0,1[} \limsup_{p \to \infty} E[(J' +\xi)(\tau^{\lambda} (\theta_p)) {{\bf 1}}_{B}] \\
 &=&E[ J'(\theta^-){{\bf 1}}_{B}]+ \sup_{\lambda \in ]0,1[} \limsup_{p \to \infty} E[\xi(\tau^{\lambda} (\theta_p)) {{\bf 1}}_{B}] \\
&\leq & E[ J'(\theta^-){{\bf 1}}_{B}]+E[ \xi(\theta){{\bf 1}}_{B}],
\end{eqnarray*}
where the last inequality follows from the inequality 
\[\limsup_{p \to \infty} E[\xi( \tau^{\lambda} (\theta_p) ){{\bf 1}}_{B}  ]= \limsup_{p \to \infty} E[\xi( \tau^{\lambda} (\theta_p) \wedge \theta){{\bf 1}}_{B}  ]\leq E[\xi(\theta) {{\bf 1}}_{B}  ],\] due to the strong left-USCE property of $\xi$ (see (\ref{uscef})).
It follows that 
\begin{eqnarray*}
E\left[ \left( J(\theta) - J(\theta^-)\right) {{\bf 1}}_{B}  \right] &\geq& E\left[ \left( J(\theta) - J'(\theta^-)- \xi(\theta) \right) {{\bf 1}}_{B}  \right].
\end{eqnarray*}

Let $B^{'}:= \left\{\,E[J'(\theta) \, |\, {\cal F}_{\theta^-}] < J'(\theta^-) \, \right\}\cap A$. 
Suppose now that we have shown that
\begin{equation}\label{empty}
B \cap B^{'} = \emptyset \;\;\; {\rm a.s.}
\end{equation}

This  yields that $B \subset (B^{'})^c$ a.s.\,, which implies that 
\[ E\left[J'(\theta) \, |\, {\cal F}_{\theta^-}\right] = J'(\theta^-)  \;\;\;{\rm  a.s.} \;\;{\rm  on} \,\, B.\]
Hence,
\begin{eqnarray*}
E\left[ \left( J(\theta) - J(\theta^-)\right) {{\bf 1}}_{B}  \right] 
&\geq& E\left[ \left( J(\theta) - J'(\theta^-)- \xi(\theta) \right) {{\bf 1}}_{B}  \right]\\
& =& \quad E\left[ \left( J(\theta) - J'(\theta)- \xi(\theta) \right) {{\bf 1}}_{B}   \right] \geq 0  \;\;\;{\rm  a.s.}\,,
\end{eqnarray*}
since $J(\theta) \geq J'({\theta})  + \xi({\theta})$  a.s.
Hence, $J$ is left-USCE. By similar arguments, we have that $J'$ is also left-USCE.

It remains to prove (\ref{empty}).
Let $C:= B \cap B^{'}$. By Lemma \ref{Propriete 3}, since $\xi$ and $-\zeta$ are strong left-USCE, we have
\begin{eqnarray*}
E[ J(\theta^-){{\bf 1}}_C] &=&  E[ J'(\theta^-){{\bf 1}}_C]+ \sup_{\lambda \in ]0,1[} \limsup_{p \to \infty} E[\xi(\tau^{\lambda} (\theta_p)) {{\bf 1}}_{C}] \\
&\leq & E[ J'(\theta^-){{\bf 1}}_{C}]+E[ \xi(\theta){{\bf 1}}_{C\cap \{\theta< T\}}] + E[ \xi(T^-){{\bf 1}}_{C\cap \{\theta=T\}}] ;
\end{eqnarray*}
\begin{eqnarray*}
E[ J'(\theta^-){{\bf 1}}_C] &=&  E[ J(\theta^-){{\bf 1}}_C]- \inf_{\lambda \in ]0,1[} \liminf_{p \to \infty} E[\zeta(\sigma^{\lambda} (\theta_p)) {{\bf 1}}_{C}] \\
&\leq & E[ J(\theta^-){{\bf 1}}_{C}]-E[ \zeta(\theta){{\bf 1}}_{C\cap \{\theta< T\}}] - E[ \zeta(T^-){{\bf 1}}_{C\cap \{\theta=T\}}] .
\end{eqnarray*}
%where the last inequality follows from the inequality 
%$$\liminf_{p \to \infty}  E[\zeta( \sigma^{\lambda} (\theta_p)} \wedge \theta){{\bf 1}}_{C}  ]\geq E[\zeta(\theta) {{\bf 1}}_{C}  ],$$ due to the left-LSCE property of $\zeta$ in the sense (\ref{ }).
By adding the two above inequalities, we get 
\[0 \leq E[ (\xi(\theta)- \zeta(\theta)){{\bf 1}}_{C\cap \{\theta< T\}}]+  E[(\xi(T^-)- \zeta(T^-)){{\bf 1}}_{C\cap \{\theta=T\} }],\]
which, with the inequality $\zeta \geq \xi $\,, leads to $\xi(\theta)=\zeta(\theta)$ a.s.\,on $C\cap \{\theta< T\}$ and $\xi(T^-)=\zeta(T^-)$ a.s.\,on $C\cap \{\theta= T\}$. Since $\xi(T^-)<\zeta(T^-)$  a.s. and since, by Lemma 
\ref{DM}, 
$\{ \xi(\theta) = \zeta(\theta), \theta<T \} \cap A = \emptyset$ a.s.\,, it follows that $P(C)=0$. \end{proof}

\begin{remark} \label{strongMM} Note that, by similar arguments as those used in the proof of Proposition \ref{Mo} and since $ 0 \leq J'  \leq  J + \xi^-$ and $ 0\leq J \leq  J'  - 
\zeta^+$, the following property holds:
when $\xi\in \S^-$ and $\zeta\in \S^+$ are integrable, the family $J$ is uniformly integrable if and only if the family $J'$ is uniformly integrable.
An additional property is provided in the next section.
\end{remark}
\section{Complementary results}\label{section.moko}

\subsection{About Mokobodski's condition}
 Proposition \ref{Y} shows that condition $J(0)<+\infty$ implies that Mokobodski's condition holds.  Indeed $J$ and $J'$ are then two nonnegative a.s. finite supermartingales such that $\xi\le J-J'\le\zeta$. In the present section, we prove that these conditions are in fact equivalent.

We first show that 
the families $J$ and $J'$ can be characterized as follows.
\begin{proposition} The families $ J $ and $ J'$ are minimal in the following sense: if $H$ and $H'$ are two nonnegative supermartingale families such that $H \geq H' + \xi $ and $H' \geq H - \zeta$,
then we have $  J  \leq H$ and $ J' \leq H'$.
\end{proposition}
\begin{proof}Let $H $ and $H' $ be two nonnegative supermartingale families such that $H \geq H' + \xi$ and 
$H' \geq H - \zeta$. Let us first show that for each $n \in \N$\,,
\begin{equation}\label{encadre}
J_n \leq H\quad {\rm and} \quad J'_n \leq H'.
\end{equation}
by induction. It clearly holds for $J_0$ and $J'_0$. Let us suppose that, for some fixed $n \in \N$, inequalities (\ref{encadre}) hold. Using the inequality $H' + \xi$ $\leq$ $H$,  we thus derive that $J'_n + \xi \leq H' + \xi$ $\leq$ $H$ . 
Since the operator $\mathcal{R}$ is non decreasing, we get $J_{n+1}= \mathcal{R}(J'_n + \xi  ) \leq \mathcal{R}( H  )$. Now, since $H$ is a supermartingale, 
by the second assertion of Proposition \ref{pro.snell}, we have $\mathcal{R}(H)=H$, and hence $J_{n+1} \leq  H  $.
By similar arguments, we also have $J'_{n+1}\leq H' $, which ensures that Property (\ref{encadre}) holds at rank $n+1$.\\
%autre preuve: We have $J'_n + \xi \leq G+\xi \leq H$. By the characterization of the Snell envelope, this implies that 
% ${\cal R } (J'_n + \xi) \leq H$.
By letting $n$ tend to $+\infty$ in (\ref{encadre}), we get that $ J  \leq H$ and $ J' \leq H'$, which ends the proof.
\end{proof}
This proposition, together with Proposition \ref{Y}, yields the following equivalence property.
\begin{proposition}\label{Mo}The condition $J(0) < +\infty$ (or equivalently $J'(0) < +\infty$) is equivalent to the {\em Mokobodski's condition}, that is, there exist two nonnegative a.s. finite supermartingale families $H$ and $H'$  such that
$$ \xi \leq H-H' \leq \zeta.$$
\end{proposition}
%Also, using the inequalities $ J -\xi \geq  J' \geq 0$ and $ J' + \zeta \geq  J \geq 0$ and the integrability assumptions made on $\xi$ and $\zeta$, one can show the following property. 
%which concerns the uniform integrability assumption made on $J$, in the existence result (see 
%Corollary \ref{existence}).
%
We also provide the following property, which completes Remark \ref{strongMM}.
\begin{proposition} \label{strongM}
The following conditions are equivalent
\begin{itemize}
\item  The family $J$ is uniformly integrable.
\item  The family $J'$ is uniformly integrable.
%\begin{eqnarray}\label{integrability}
%E[\esssup_{\theta \in \T_0}  J  (\theta)] < + \infty &{\rm and} &E[\esssup_{\theta \in \T_0}  J' (\theta)] < + \infty.
%\end{eqnarray}
\item 
The {\em strong Mokobodski condition} holds that is, there exist two nonnegative supermartingale families $H$ and $H'$ 
uniformly integrable such that 
\[ \xi \leq H-H' \leq \zeta.\]
%\begin{eqnarray*}
%E[\esssup_{\theta \in \T_0}  H  (\theta)] < + \infty &{\rm and} &E[\esssup_{\theta \in \T_0}  H' (\theta)] < + \infty.
%\end{eqnarray*}
\end{itemize}
\end{proposition}

\begin{remark}
Let us note that a weak point of the Mokobodski condition, is that it is quite difficult to check. Our approach is more constructive, since $J$ and $J'$ are always well defined as the nondecreasing limits of nonnegative supermartingale families.
\end{remark}

\subsection{The right-continuous in expectation case} 
Let us now introduce the following definition. 
\begin{definition}
An admissible family $\left( \phi ( \theta ), \theta \in \T  \right)$ is said to be \emph{right-continuous in expectation along stopping times (right-CE)} if for all $\theta \in \T$ and for all sequences of stopping times $ (\theta_n)_{n \in \N}$ such that $\theta^n\downarrow \theta$,
%(resp. $ \theta^n \uparrow \theta$) 
$E[\phi(\theta) ]=\lim_{n\to \infty} E[\phi(\theta_n)]$.

\end{definition}
We first show that the limit of a non decreasing sequence of right-CE supermartingale families is also a right-CE supermartingale family. More precisely, the following property holds.
\begin{lemma}\label{limitsuperbis}Let  $(\phi_n)_{n \in \N}$ be a non decreasing sequence of  right-CE nonnegative supermartingale families.
%INUTILE ICI puisque positif:Suppose also that for each $\theta$ $\in$ $T_0$, $E[ \esssup_{n \in \N} \phi_n(\theta)^- ] < + \infty$.\\
The family $\phi$ defined for each $\theta$ $\in$ $\T$ by $\phi(\theta)= \lim_{n \to \infty} \phi_n(\theta)$ a.s. is then a right-CE supermartingale family.
\end{lemma}
\noindent Note that in order to define $\phi(\theta)(\omega)$ for each $\omega$, it suffices to set $\phi(\theta)(\omega):= \limsup_{n \to \infty} \phi_n(\theta)(\omega)$
%\noindent{\sc Proof of Lemma \ref{limitsuper}:}
\begin{proof}{} By Lemma \ref{limitsuper}, we already know that  the family $\phi$ is a supermartingale family. It remains to show that it is  right-CE.
Let $\theta \in \T$ and let $ (\theta_p)_{p \in \N}$ be a sequence of stopping times such that $\theta_p\downarrow \theta$.  
By the monotone convergence theorem and the right-CE property of $\phi_n$, we have
\begin{align*}
\lim_{p\to \infty} \uparrow E[\phi(\theta_p)] &= \lim_{p\to \infty} \uparrow E[ \lim_{n\to \infty} \uparrow \phi_n(\theta_p)]\\
&=  \lim_{p\to \infty} \uparrow  \lim_{n\to \infty}  \uparrow E[\phi_n(\theta_p)]=  \lim_{n\to \infty} \uparrow  \lim_{p\to \infty}  \uparrow E[\phi_n(\theta_p)]\\
&=  \lim_{n\to \infty} \uparrow   E[\lim_{p\to \infty}  \uparrow \phi_n(\theta_p)]= E[\phi(\theta) ].
\end{align*}  
The proof is thus complete.
\end{proof}

\begin{remark}
Note also that, using the above lemma together with the well-known result of aggregation of right-CE supermartingales (see Th. 3.13 in Karatzas and Shreve (1994)), one can easily derive the analogous result for processes (see Th. 18 ch. VI in Dellacherie Meyer (1980) and its quite technical proof). 
\end{remark}

\begin{proposition}
Suppose that $J (0)< + \infty$ and that the families $\xi$ and $-\zeta$ are right-CE.
Then, the families $J$ and $J'$ are right-CE. Also, $Y(=J-J')$ is the common value function of the game problem and is right-CE.
\end{proposition}
\begin{proof}{}
Recall that by a classical result of optimal stopping theory (see Lemma 2.13 in El Karoui (1981)), the value function associated with a right-CE reward family is right-CE. This ensures that by induction, for each $n$, $J_n$ and $J'_n$ are right-CE. Since $J =   \lim_{n\to \infty} \uparrow J_n$ and 
$J' =   \lim_{n\to \infty} \uparrow J'_n$, by Lemma \ref{limitsuperbis}, $J$ and $J'$ are  right-CE.\\
Moreover, by Theorem \ref{saun}, since $\xi$ and $-\zeta$ are right-CE and hence right-USCE, it follows that $Y=  \underline V = \overline V$ that is, $Y$ is the common value function of the game problem. Also, since $Y=J-J'$, it is  right-CE.
\end{proof}

\subsection{Remarks  about some works related to Dynkin Games}

In \cite{LM}, Lepeltier and Maingueneau do not suppose that  Mokobodski's  condition holds. However, the game is proven to have a value under the stronger regularity hypothese that   $(\xi_t)$ and $(\zeta_t)$ are right-continuous.  Note also that their approach does not provide a construction of the common value function. Let us underline that this approach relies on very sophisticated results of the 
General Theory of Processes and Optimal Stopping Theory, even stronger than in \cite{A}.\\
\indent The Dynkin  game problem and its links with reflected backward stochastic differential equations  (RBSDEs) has also received much attention  (see for example \cite{CK}, \cite{HHO} in the case of a Brownian filtration). These  authors suppose that $\xi < \zeta$ but do not suppose that  Mokobodski's condition holds; they provide an existence result for the RBSDE, via a penalization method. When the coefficient of the RBSDE is equal to zero, this result ensures that Mokobodski's condition holds.
\\
\indent In order to complete this brief review, note that in \cite{TV},  Touzi  and Vieille study a Dynkin Game of a different nature, as the set of controls is larger than the set of stopping times. Indeed, one can easily find some examples of  games that are fair in their case, but which are not  in the case when the controls are given by stopping times.

\appendix
\section{Appendix}
\subsection{Some results on optimal stopping in the framework of admissible families}\label{subsec.opt}

\noindent Let $\phi= (\phi(\theta),\theta\in \T)$ be an admissible family called {\em reward} in $\mathcal{S}^-$, that is, satisfying the following integrability condition:
\begin{equation}\label{IC}
E[\esssup_{\theta\in \T} (\phi(\theta))^{-} ] < + \infty.
\end{equation}

For $\theta \in \T$, the {\em value function at time $\theta$} is given by
\begin{equation*}
v(\theta)=\esssup_{\tau \in \T_{\theta}} E[\phi(\tau) \, |\,\F_{\theta}].
\end{equation*}
%where $T_{S^+}$ is the class of stopping times
%$\theta\in T_0$ with $\theta>S $ a.s. on $\{S<T\}$ and $\theta=T$ a.s. on $\{S=T\}$.
%Note that $\vp(\theta)= \phi (T)$ a.s. on $\{S=T\}$.

This optimal stopping problem clearly reduces to the case of a nonnegative reward, which has been studied by Kobylanski and Quenez (2011) in the framework of families of random variables.

More precisely, define $X(\theta):=E[\esssup_{\tau \in \T} (\phi(\tau))^- \, |\,\F_{ \theta }]$ and $\bar \phi (\theta):= \phi(\theta) + X(\theta)$.
Note that  $X=(X(\theta),\theta\in \T)$ is a martingale family. The new reward family $\bar \phi$ is nonnegative and the associated new value function $\bar v$ satisfies
$\bar v(\theta)= v(\theta) + X(\theta)$ a.s.\,The optimal stopping problem associated with the reward $\bar \phi$ can be thus solved by translation.

\begin{proposition}\label{prop.snell} Let $\phi\in \mathcal{S}^-$.
The value function  family $v$ associated to $\phi$ is equal to the Snell envelope of $\phi$, that is  $v=\mathcal{R}(\phi)$.
\end{proposition}

 The following proposition gives some useful properties of the Snell envelope operator.
 \begin{proposition}\label{pro.snell}
  The Snell envelope operator $\mathcal{R}$  is non decreasing on $\mathcal{S}^-$. \\
  For each $X\in \S^-$\,,  $\mathcal{R}(X)\ge X$, and equality holds if and only if $X$ is a supermartingale family.
  \\
  If the family $X$ is uniformly integrable, then $\mathcal{R}(X)$  is uniformly integrable.
\end{proposition}

The following property, known as the {\em optimality criterion} holds true.
\begin{proposition}\label{prop.opt}
Let $\theta$ $\in$ $\T$ and let $\hat \tau \in \T_{\theta}$ be such that $E[\phi(\hat \tau)]<\infty$.\\ 
The stopping time $\hat \tau$ is $\theta$-optimal, that is, $v(\theta) = E[\phi( \hat \tau) \, |\,\F_{\theta}]$ a.s.\,,
if and only if 
\begin{equation*}
v ( \hat \tau) = \phi( \hat \tau) \quad  {\rm a.s.} \quad  {\rm and} \quad ( v( \tau), \tau \in \T_{[\theta,\hat \tau]}) \quad  {\rm is \,\,a \,\,martingale \,\,family} .
\end{equation*}
\end{proposition}

The existence of $\varepsilon$-stopping time is provided under a right regularity condition. Let us introduce, for each $\theta$ $\in$ $\mathcal{T}$ and $\lambda$ $\in$ $]0,1[$,
the stopping time $\tau^{\lambda} (\theta)$ defined by 
\begin{equation} \label{taulambda}
 \tau^{\lambda} (\theta):=
\essinf \{\,\tau \in \mathcal{T}_{\theta}\, , \,\lambda v(\tau) \leq\phi(\tau) \,\,\mbox {a.s.} \,\}
\end{equation}
Note that, by definition, $\tau^{\lambda} (\theta)$ is non decreasing with respect to $\lambda$. Moreover, it is also non decreasing with respect to $\theta$, that is, for all $\theta, \theta'$ $\in \mathcal{T}$, we have $\tau^{\lambda} (\theta)$ $\geq$ $\tau^{\lambda} (\theta')$ a.s. on $\{ \theta \geq \theta'\}$.

\begin{theorem}\label{thm.epsopt}Suppose the reward family $\phi$ is right-USCE, $v$ is non negative and  
 $v(0)<\infty$.\\
For each $\theta$ $\in$ $\mathcal{T}$ and $\lambda$ $\in$ $]0,1[$, the stopping time $\tau^{\lambda} (\theta)$ satisfies 
\begin{equation}\label{lambdai}
\lambda v(\tau^{\lambda} (\theta) )\leq  \phi(\tau^{\lambda} (\theta) )\,\,\mbox {a.s.}
\end{equation}
Also, 
$\displaystyle \left( v(\tau), \tau \in \mathcal{T}_{[\theta, \tau^\lambda(\theta)]}\right)$ is a martingale family.\\
 In particular,  $\tau^{\lambda} (\theta)$ is $(1- \lambda)$-optimal for  $v(\theta)$, that is, $\lambda v(\theta) \leq  E[\phi \left( \tau^{\lambda} (\theta)\right)\, |\,
 {\cal F}_{\theta}]$ a.s.
\end{theorem}

\begin{proof}{} Let us introduce the \emph{strict value function at time $\theta$} associated with $\phi$, defined by
\begin{equation} \label{vplus}
v^+(\theta):={\rm ess}\sup_{\tau \in \mathcal{T}_{\theta^+}} E[\phi(\tau) \, |\,\mathcal{F}_{\theta}] \,.
\end{equation}
where $\mathcal{T}_{\theta^+}$ is the class of stopping times
$\tau \in \mathcal{T}_0$ with $\tau>\theta $ a.s. on $\{\theta<T\}$ and $\tau=T$ a.s. on $\{\theta=T\}$.
Also, we denote by $\bar v^+(\theta)$ the \emph{strict value function at time $\theta$} associated with $\bar \phi$.
First, by Proposition 1.9 in \cite{KQ}, we have $\bar v= \bar v^+ \vee \bar \phi$. Also, $\bar v^+= v^+ + X$ and $ v=  v^+ \vee  \phi$. 
By Proposition 1.12 in \cite{KQ}, $\bar v^+$ is right-CE. Now, $X$ is clearly right-CE since it is a martingale family. Hence, as $v^+ = \bar v^+ - X$, it follows that $v^+$ is right-CE. These properties, together with the same arguments as those used in the proof of Lemma 2.5 in \cite{KQ},
lead to inequality 
(\ref{lambdai}). Moreover, since $v \geq 0$, the same proof as that of Lemma 2.7 in \cite{KQ} shows that 
$$ v(\theta) = E [ v(\tau^{\lambda} (\theta)) \, |\, \mathcal{F}_{\theta}] \,\,\mbox {a.s.},$$
which ensures the second assertion of the theorem.
The last assertion follows from the two previous ones, by taking the conditional expectation given ${\cal F}_{\theta}$ in inequality (\ref{lambdai}).
\end{proof}

\begin{remark} 
When $v$ is not necessarily non negative, the above theorem does not hold. However, by Theorem 2.1 in \cite{KQ} applied to 
the non negative reward family $\bar \phi= \phi + X$ and its associated value function $\bar v= v+X$, the stopping time $\bar \tau^{\lambda} (\theta)$ defined, for each 
$\lambda \in ]0,1[$, by 
\begin{equation}\label{barretau}
\bar \tau^{\lambda} (\theta):=
\essinf \{\,\tau \in \mathcal{T}_{\theta}\, , \,\lambda \bar v(\tau) \leq \bar \phi(\tau) \,\, \mbox{ a.s.} \,\}
\end{equation}
satisfies the inequality $\lambda \bar v(\bar \tau^{\lambda} (\theta) )\leq  \bar \phi(\bar \tau^{\lambda} (\theta) )$ a.s.
%Also, 
%$ \displaystyle \left( \bar v(\theta), \theta \in \mathcal{T}_{[S, \bar \tau^\lambda(\theta)]}\right)$ is a martingale family.
Also, the inequality  $\lambda \bar v(\theta) \leq  E[\bar \phi \left(\bar \tau^{\lambda} (\theta)\right)\, |\,
 {\cal F}_{\theta}]$ holds a.s.\,, which yields that 
 $\lambda  v(\theta) \leq  E[ \phi \left( \bar \tau^{\lambda} (\theta)\right)\, |\,
 {\cal F}_{\theta}] + (1- \lambda) X(\theta)$ a.s.\, The stopping times $\bar \tau^{\lambda} (\theta)$, $\lambda \in ]0,1[$, can thus be also interpreted as $\epsilon$-stopping times, but they appear less tracktable than the stopping times $\tau^{\lambda} (\theta)$, $\lambda \in ]0,1[$.
  \end{remark}
\begin{remark} The above theorem can be applied, in the Dynkin game problem, to the value functions $J$ and $J'$ because they are non negative.
 \end{remark}
The existence of $\theta$-optimal stopping times is proven under additional left regularity. More precisely
\begin{theorem}\label{thm.exopt}
If the reward $\phi$ is USCE (that is right- and left-USCE) and $v=\mathcal{R}(\phi)$ satisfies $v(0)<\infty$, then, for each $\theta$ $\in$ $\mathcal{T}$\,,
the stopping time $ \tau_{*}(\theta)$ defined by
\begin{equation}\label{T.TAO}
 \tau_{*}(\theta):=\essinf\{\tau \in\mathcal{T}_{\theta}\,, v(\tau)= \phi(\tau) \; \mbox{ a.s.} \,\}
\end{equation}
is the minimal optimal stopping time  for $v(\theta)$. 
\end{theorem}
\begin{proof}
By the same arguments as those used in the proof of Theorem 2.9 in \cite{KQ}, applied here to the optimal stopping problem associated to reward $ \bar \phi$, and since 
$\tau_{*}(\theta)=\essinf\{\tau \in\mathcal{T}_{\theta}\,, \bar v(\tau)= \bar \phi(\tau) \; {\rm a.s.} \,\}$, we derive that 
$\tau_{*}(\theta)$ is $\theta$-optimal for $\bar v(\theta)$, and hence for $v(\theta)$.
\end{proof}
\begin{remark}
We also have $\tau_{*}(\theta)= \lim_{\lambda\uparrow 1} \uparrow  \bar \tau^\lambda(\theta)$ a.s.\,, 
where  $ \bar \tau^\lambda(\theta)$ is defined by (\ref{barretau}).
If $v \geq 0$, then $\tau_{*}(\theta)= \lim_{\lambda\uparrow 1} \uparrow  \tau^\lambda(\theta)$ a.s.\, where 
 $\tau^\lambda(\theta)$ is defined by (\ref{taulambda}).
\end{remark}
At last, recall some regularity results of supermartingale families, which are used in the proof of Theorem \ref{left-USCE}. First, one can easily show that a supermartingale family is right-USCE. 

%\begin{proposition}\label{vLCE}
%If $(\phi(\theta),\theta\in \T)$ is USCE  and $v(0)<\infty$, then $(v(\theta),S\in \T)$ is 
%left-continuous in expectation (LCE). 
%\end{proposition}

Let now $\theta$ in $\T$. Recall that a non decreasing sequence of stopping times $(\theta_n)_{n \in \N}$ such that for each $n$\,, $\theta_n \leq \theta$ a.s.\,, is said to {\em announce} $\theta$ on $A$ $\subset$ $\Omega$ if 
\[ \theta_n \uparrow \theta \mbox{ a.s. on }A \; \mbox{ and }\theta_n<\theta \mbox{ a.s. on }A .\]
%Let $ {\cal A}(S,A)$ be the set of all sequences of stopping times that announce $\theta$ on $A$. Recall that $\theta$ is said to be accessible on $A$ if  ${\cal A}(S,A)$ is not empty. 
The following lemma, which is used in the proof of Theorem \ref{left-USCE}, holds. 
\begin{lemma} \label{sull}
Let $\big(u(\theta), \theta\in \T\big)$ be a supermartingale family in $\mathcal{S}^-$.\\ 
Let $\theta$ in $\T$. Suppose that $\theta$ is accessible on a subset $A$ of $\Omega$.\\
There exists an $\F_{\theta^-}$-measurable random variable $u(\theta^-)$, unique on $A$ (up to the equality a.s.\,), such that, for any non decreasing sequence $(\theta_n)_{n\in \N}$ announcing $\theta$ on $A$, one has
\[u(\theta^-)=\lim_{n\to\infty} u(\theta_n) \quad \mbox{a.s. on } A.\]
If $u(0) < +\infty$, then $u(\theta^-){{\bf 1}}_A$ is integrable.
\end{lemma} 

%\begin{remark}\label{acce}Recall that, by definition, the set $A(\theta)$ of accessibility of $\theta$ is the union of the sets on which $\theta$ is accessible. By a result of Dellacherie and Meyer (Chap IV.80), there exists a sequence of sets $(A_k)_{k \in \N}$ in ${\cal F}_{S^-}$ such that  for each $k$, $\theta$ is accessible on $A_k$ and $A(\theta) = \cup_k A_k$ a.s.\,It follows that $u(S^-)$ is well defined on $A(\theta)$ and depends only on $\theta$. Actually, this property is not used in this paper.
%\end{remark}  

\begin{proof}{} This lemma corresponds to Lemma 4.8 in Kobylanski and Quenez (2011). For the convenience of the reader, we give the sketch of the proof.
 Without loss of generality, 
we can suppose that $u$ is nonnegative. Let $(\theta_n)_{n\in \N}$ be a non decreasing sequence  announcing $\theta$ on $A$. 
It is clear that $(u( \theta_n))_{ n\in \N}$ is a discrete nonnegative supermartingale relatively to the filtration $(\F_{\theta_n})_{n\in \N}$. By the well-known convergence theorem for discrete supermartingales, there exists a random variable $Z$ such that  
$(u( \theta_n))_{ n\in \N}$ converges a.s. to  $Z$. If $u(0) < +\infty$, then $Z{{\bf 1}}_A$ is integrable. We then set $u(\theta^-) := Z$.
It remains to show that this limit $u(\theta^-)$, restricted to $A$, does not depend on the sequence $(\theta_n)_{n \in \N}$. For this proof, one is refered to Kobylanski and Quenez (2011).
\end{proof}

We now complete the properties of $u(\theta^-)$. Let $\theta$ in $\T$. Let us precisely define $A(\theta)$, the union of the sets on which $\theta$ is accessible, called the {\em set of accessibility} of $\theta $  (see \cite{DM2}).\\
More precisely, 
for each non decreasing sequence of stopping times $(\theta_n)_{n \in \N}$, also denoted by $(\theta_n)$, we set $A[(\theta_n)]:= \{ \theta_n \uparrow \theta \; \mbox{ and }\theta_n<\theta \,\, \mbox{for all}\,\,n\, \}$.
Let $\mathcal{A}_{\theta}$ be the set of sequences $(\theta_n)_{n\in \N}$ such that $A[(\theta_n)]$ is non empty. If $\mathcal{A}_{\theta}$ is empty, we set $A(\theta) := \emptyset$. Otherwise,
we introduce the following random variable 
$$X:= {\rm ess}\sup_{(\theta_n) \in \mathcal{A}_{\theta}} {\bf 1}_{ A[(\theta_n)] }.$$
By classical results on the essential supremum (see Proposition VI.1.1  p121 in \cite{Neveu} ), there exists a sequence of elements of $\mathcal{A}_{\theta}$, denoted by $\{ (\theta^k_n)_{n \in \N} \,, \,k \in \N \}$, such that 
$X=\sup_{k \in \N} {\bf 1}_{ A[(\theta^k_n)]}$ a.s.\, The set $A(\theta)$ can then be clearly defined by 
$A(\theta) := \cup_{k \in \N} A[(\theta^k_n)]$ and belongs to ${\cal F}_{\theta^-}$. This together with the above lemma leads to the following result, which corresponds to Theorem 4.3  in \cite{KQ}.

\begin{theorem} 
A supermartingale system $(u(\theta), \theta\in T)$ in $\mathcal{S}^-$ is left-limited along stopping times at each stopping time $\theta$ $\in$ $T_{0^+}$, that is, there exists an ${\cal F}_{\theta^-}$-measurable random variable $u(\theta^-)$ such that, for any non decreasing sequence  of stopping times  $(\theta_n)_{n\in \N}$,
\[\phi(\theta ^-)=\lim_{n \to \infty} \phi(\theta_n) \mbox{ a.s. on }A[(\theta_n)]. \]
If $u(0) < +\infty$, then $u(\theta^-){\bf 1}_{A(\theta)}$ is integrable.\end{theorem}

\subsection{Two useful lemmas}

We now provide two lemmas which are used in the proof of Theorem \ref{left-USCE}. 

Let $\phi =(\phi(\theta), \theta\in \T)$ be a right-USCE in $\S^-$, set  $v = {\cal R}(\phi)$, and 
suppose $v(0) <  \infty$. 
 
Let $\theta$ $\in$ $\T$ and let $(\theta_p)_{p \in \N}$ in $\T$ such that $\theta_p \uparrow \theta$.

Suppose that the event $A: = \{\theta_p<\theta,\,\, \mbox{for all}\,\,p\, \}$ is non empty and define the event $B$ by  
$B:=\{ E [ v({\theta})  \, |\, {\cal F}_{\theta^-} ] < v({\theta^-})\} \cap A$. Now, the following lemma holds.
\begin{lemma} \label{Propriete 2}
The sequence of stopping times $(\tau^{\lambda} (\theta_p))_{p \in \N}$ announces $\theta$ on $B$.
\end{lemma}
 \begin{proof}{} This argument is detailed inside the proof of Theorem 4.16 in \cite{KQ}. For the reader's comfort we provide it in full extension.
Let us first show that for each $p $ and for each $\lambda \in ]0,1[$, $\tau^{\lambda} (\theta_p) < \theta$ a.s. on $B$, or equivalently that $B \cap \{\tau^{\lambda} (\theta_p) \geq \theta\} = \emptyset$ a.s.\\
Note first that $\displaystyle{\{\tau^{\lambda} (\theta_p) \geq \theta\} = \cap_q \{\tau^{\lambda} (\theta_p) \geq \theta_q \}}$. Hence, $\{\tau^{\lambda} (\theta_p) \geq \theta\}$ $\in$ $ \mathcal{F}_{\theta^-}\cap \vee_n \mathcal{F}_{\theta_n}$.\\
Now, since  $\displaystyle\left(v(\tau), \tau \in \mathcal{T}_{[ \theta_p,   \tau^{\lambda} (\theta_p)]} \right)$ is a martingale family, it follows that for each $q\geq p$\,,\\
$E[ v(\tau^{\lambda} (\theta_p))\, |\, \mathcal{F}_{\theta_q} ]= v(\theta_q)$ a.s. on $\{\tau^{\lambda} (\theta_p) \geq \theta_q \}$ and hence  on $\{\tau^{\lambda} (\theta_p) \geq \theta \}$. Therefore, by letting $q$ tend to $\infty$, 
\[
E[ v(\tau^{\lambda} (\theta_p))\, |\, \vee_n \mathcal{F}_{\theta_n} ]= v({\theta^-}) \quad {\rm a.s.}\,\, {\rm on} \quad \{\tau^{\lambda} (\theta_p) \geq \theta \}\cap  A.\]
Now, as $(\theta_p)$ announces $\theta$ on $A$, by a classical measurability property (see Lemma A1 in \cite{KQ}), we have
\begin{equation*}\label{uuuter}
E[ v(\tau^{\lambda} (\theta_p))\, |\, \vee_n \mathcal{F}_{\theta_n} ]= E[ v(\tau^{\lambda} (\theta_p))\, |\, \mathcal{F}_{\theta^-} ] \quad {\rm a.s.}\,\, {\rm on}\,\, A.
\end{equation*}
It follows that, on the one hand,
\begin{equation}\label{uuu}
E[ v(\tau^{\lambda} (\theta_p))\, |\, \mathcal{F}_{\theta^-} ]= v({\theta^-}) \quad {\rm a.s.}\,\, {\rm on} \quad \{\tau^{\lambda} (\theta_p) \geq \theta \}\cap  A.
\end{equation}
On the other hand, since $\displaystyle \left(v(\tau), \tau \in \mathcal{T}_{[ \theta_p,   \tau^{\lambda} (\theta_p)]} \right)$ is a martingale, we have $ E[ v(\tau^{\lambda} (\theta_p))\, |\, \mathcal{F}_{\theta} ]= v(\theta)$ a.s. on $\{\tau^{\lambda} (\theta_p) \geq \theta\}$. By taking the condidional expectation with respect to $\mathcal{F}_{\theta^-}$, we derive that 
\[ E[ v(\tau^{\lambda} (\theta_p))\, |\, \mathcal{F}_{\theta^-} ]= 
E[ v(\theta)\, |\, \mathcal{F}_{\theta^-} ] <  v(\theta^-) \quad {\rm a.s.}\,\, {\rm on} \quad B \cap \{\tau^{\lambda} (\theta_p) \geq \theta\},\]
which, with equality (\ref{uuu}), yields that $B\cap \{\tau^{\lambda} (\theta_p) \geq \theta\} = \emptyset$ a.s.

It follows that for each $p$, $\theta_p \leq \tau^{\lambda} (\theta_p) < \theta$ a.s.\, on $B$. Hence, $\tau^{\lambda} (\theta_p) \uparrow \theta$ a.s. on $B$ as $p$ tends to $+ \infty$. In other words,  the sequence $(\tau^{\lambda} (\theta_p))_{p \in \N}$ announces $\theta$ on $B$. 
\end{proof}

\begin{lemma} \label{Propriete 3} 
 If the family $\phi$ is uniformly integrable, we have
\begin{equation*}
E[ v(\theta^-) {{\bf 1}}_{B}] =   \sup_{\lambda \in ]0,1[} \limsup_{p \to \infty} E[\phi(\tau^{\lambda} (\theta_p)) {{\bf 1}}_{B}].
\end{equation*}
Also, for each $D \in {\cal F}_{\theta^-}$\,, this equality holds with $B$ replaced by $B \cap D$.
\end{lemma}
\begin{proof}{} Let $\lambda \in ]0,1[$ and let $p \in \N$. By (\ref{lambdai}), we have
$
\lambda v(\tau^{\lambda} (\theta_p) ) \leq \phi(\tau^{\lambda} (\theta_p) )$ a.s.\,
It follows that 
\[\lambda E[v(\tau^{\lambda} (\theta_p)){{\bf 1}}_{B} ] \leq  E[\phi(\tau^{\lambda} (\theta_p)) {{\bf 1}}_{B}].\]
By letting $p$ tend to $+ \infty$, we derive that
\[\lambda  \limsup_{p \to \infty} E[v(\tau^{\lambda} (\theta_p) ) {{\bf 1}}_{B}]\leq \limsup_{p \to \infty} E[\phi(\tau^{\lambda} (\theta_p)) {{\bf 1}}_{B}].\]
Now, by the above lemma, the sequence $(\tau^{\lambda} (\theta_p))_{p \in \N}$ announces $\theta$ on $B$. Hence, by Lemma \ref{sull}, it follows that $\lim_{p\to\infty} v(\tau^{\lambda} (\theta_p))= v(\theta^-)$ a.s. on $B$. 
Using Fatou's lemma, we get 
\[ E[v(\theta^{-} ) {{\bf 1}}_{B}] \leq  \liminf_{p \to \infty} E[v(\tau^{\lambda} (\theta_p) ) {{\bf 1}}_{B}] \leq  \limsup_{p \to \infty} E[v(\tau^{\lambda} (\theta_p) ) {{\bf 1}}_{B}].\]
The two above inequalities yield that 
\[\lambda  E[v(\theta^{-} ) {{\bf 1}}_{B}] \leq \limsup_{p \to \infty} E[\phi(\tau^{\lambda} (\theta_p)) {{\bf 1}}_{B}],\]
and this holds for each $\lambda \in ]0,1[$. Hence, by taking the supremum over $\lambda \in  ]0,1[$, we get 
\[E[ v(\theta^-) {{\bf 1}}_{B}] \leq   \sup_{\lambda \in ]0,1[} \limsup_{p \to \infty} E[\phi(\tau^{\lambda} (\theta_p)) {{\bf 1}}_{B}].\]
The same arguments show that for each $D \in {\cal F}_{\theta^-}$, this inequality holds with $B$ replaced by $B \cap D$.\\
At last, by using the inequality $v \geq \phi$ and the fact that $v$ is uniformly integrable, we derive that the inequality is an equality. 
\end{proof}

\section{Case of payoffs given by processes}
In this section, by using the results provided in this paper, we derive the corresponding results in the case of processes.  We underline that if a progressive process $(\phi_t)$ gives naturally way to an admissible family $(\phi_\theta, \theta \in \mathcal{T})$, the converse is not always possible. Indeed for a given admissible family $(\phi(\theta),  \theta \in \mathcal{T})$, there does not 
always exist a progressive process $(\phi_t)$ which {\em aggregates} $(\phi(\theta),  \theta \in \mathcal{T})$, that is, such that $\phi_\theta=\phi(\theta)$ a.s. for all $\theta \in \mathcal{T}$. The main point is thus to prove the existence of progressive processes $(J_t)$ and $(J'_t)$ that aggregate the families $(J(\theta), \theta\in \mathcal{T})$ and $(J'(\theta), \theta\in \mathcal{T})$.

Let  $ (\xi_t)$ and $(\zeta_t)$ be  progressive processes such that $\xi_T=\zeta_T=0$ and  such that 
\[  E\left[\esssup_{t\in[0, T]} \xi_t^-\right] <+\infty\quad \mbox{and} \quad E\left[\esssup_{t\in [0,T]} \zeta_t^+\right]<+\infty.   \]
As above, we can define the families $(J_n(\theta), \theta\in \mathcal{T})$ and $(J'_n(\theta), \theta\in \mathcal{T})$ as well as  $(J(\theta), \theta\in \mathcal{T})$ and $(J'(\theta), \theta\in \mathcal{T})$, associated to the admissible families $(\xi_\theta, \theta \in \T)$ and $(\zeta_\theta, \theta \in \T)$. Suppose now that $J(0)<+\infty$. Using Proposition B.1. in \cite{KQ}, one can show by induction that, for each $n$, there exist supermartingale processes  $(J_{n,t})$ and $(J_{n,t}')$ such that, for all $\theta\in \T$, $J_{n,\theta}=J_n(\theta)$ and  $J'_{n,\theta}=J'_n(\theta)$ a.s.  We can then define the processes $(J_t)$ and $(J'_t)$ by $J_t:=\lim_{n\to \infty}\uparrow J_{n,t}$ and $J'_t:=\lim_{n\to \infty}\uparrow J'_{n,t}$. Clearly, for each $\theta\in \T$, $J_\theta=\lim_{n\to \infty}\uparrow J_{n,\theta}=\lim_{n\to \infty}\uparrow J_n(\theta)=J(\theta)$ a.s., and similarly $J'_\theta=J'(\theta)$ a.s.   %Moreover, if $J(0)<+\infty$, the supermartingale process $(J_t)$ is a strong supermartingale according to \cite{DM2}, Appendix 1, Definition 1.  By  \cite{DM2}, Appendix 1 Theorem 4,   $(J_t)$ is RLLL. The same holds  and $(J'_t)$. 
We thus have proven the following result. 
\begin{proposition}\label{prop.proc} Suppose that $J(0)<+\infty$.
There exists two  nonnegative supermartingale processes $(J_t)$ and $(J'_t)$ such that for all $\theta\in \T$, $J_\theta=J(\theta)$ and  $J'_\theta=J'(\theta)$ a.s.\,
Moreover,  
\begin{equation}\label{eee}
  J  _{ \theta } = \esssup_{\tau \in \T_{\theta}} E \left[  J' _{ \tau } + \xi _{ \tau } \given \F_{\theta} \right]\,\, \mbox{a.s.} \quad {\rm and}\quad
 J' _{ \theta } = \esssup_{\sigma \in \T_{\theta}} E \left[  J _{\sigma } - \zeta _{\sigma }\given \F_{\theta} \right] \,\, \mbox{a.s.} 
\end{equation}
\end{proposition}

This property together with Theorem \ref{saun} gives the following result.

\begin{theorem}\label{thm.B}
Suppose that $ J_0< + \infty$ and that the processes $(\xi_\theta, \theta \in \T)$ and $(-\zeta_\theta, \theta \in \T)$ are right-USCE.
Then, the game is fair and the common value function process is equal to  $Y_t:=J_t-J_t'$. 
\end{theorem}

\begin{remark}
Note that if a process $(\phi_t)$ is right-upper semicontinuous (that is, for almost every $\omega$, the function $t \mapsto \phi_t (\omega)$ is right-upper semicontinuous), then $(\phi_\theta, \theta\in \T)$ 
 is right-USC (along stopping times). If, moreover, the process  $(\phi_t)$ is of class $\mathcal{D}$ (or equivalently the family $(\phi_\tau, \tau \in \T)$ is uni\-for\-mly integrable), then  $(\phi_\theta, \theta\in \T)$  is right-USCE.  
\end{remark}

Recall that a progressive process $(\phi_t )$ is said to be of class 
$\mathcal{D}$ if the associated family of 
random variables $(\phi_\theta, \theta\in \T)$ is uniformly integrable. The following existence result holds.
\begin{theorem}\label{thm.left-USCE}   
Suppose that $(J_t)$ and $( J'_t)$ are of class $\mathcal{D}$ and that the families  $(\xi_\theta, \theta\in \T)$ and $(-{\zeta}_\theta, \theta\in \T)$ are right-USCE and strong left-USCE. Suppose also that for each predictable stopping time $\tau \in \T$,  we have $ \{ \xi_\tau = \zeta_\tau \} = \emptyset $ a.s. on $\{\tau <T\}$\, and that $(\xi_t)$ and $( \zeta_t)$ are left-limited at $T$, with
 $ \{ \xi_{T^-} = \zeta_{T^-} \} = \emptyset $ a.s. 
Then, the game is fair and, for each $\theta \in \T$\,, the pair of  stopping times $\left(  \tau_{*} (\theta),\sigma_* (\theta) \right)$, defined by (\ref{tauetoile}) and (\ref{sigmaetoile}),  satisfies
\begin{equation}\label{tt}
 \tau_{*} (\theta)=
\inf \{\,t\ge \theta\, , \,  J _t =  J'_t+ \xi_t  \,\}
\quad \mbox{and} \quad  \sigma_* (\theta)=
\inf \{\,t\ge \theta\,\, , \, J'_t =  J_t- \zeta_t\ \,\}
\end{equation}
and is a $\theta$-saddle point for the criterion $I_{\theta}$. In particular 
$ \underbar{V}_\theta = \bar{V}_\theta  = I_{\theta} \left( \tau_{*} (\theta), \sigma_* (\theta)\right)$ a.s.\\
\end{theorem}
\begin{proof} Equalities (\ref{tt}) follow from Propositions B.5 and B.6 in \cite{KQ}. The result clearly follows 
from Corollary \ref{existence}. 
\end{proof}

\begin{remark}
The result still holds if instead of the assumption : $(\xi_t)$ and $( \zeta_t)$ are left-limited at $T$, with
 $ \{ \xi_{T^-} = \zeta_{T^-} \} = \emptyset $ a.s.\, , we only suppose that  $ \{\limsup_{t \rightarrow  T^-}\xi_t =\liminf_{t \rightarrow  T^-}\zeta_t  \} = \emptyset $ a.s.\, 
Note that if a process $(\phi_t)$ is upper semicontinuous and of class $\mathcal D$, then $(\phi_\theta, \theta\in \T)$ 
 is right-USCE and strong left-USCE. 
 
 We stress on that Theorem \ref{thm.left-USCE} follows from Corollary \ref{existence}, whose proof requires no sophisticated mathematical tools, contrary to that given in the previous literature  (cf. \cite{A} Theorem 4-3 p.31). Actually, the only result of the General Theory of Processes used above is that each right-CE supermartingale family can be aggregated by an RCLL supermartingale process.
\end{remark}

 Note that equalities (\ref{eee}) can be expressed in terms of processes, by using $\hat{\mathcal{R}}$, the Snell envelope 
operator acting on progressive processes, often used in the literature on Dynkin games problems. We recall below the definition of $\hat{\mathcal{R}}$, which is not as simple as that of $\mathcal{R}$, the Snell envelope 
operator acting on admissible families. It requires to introduce some supermartingale families, and then to aggregate them so that the operator $\hat{\mathcal{R}}$ acts on processes.\\
\indent First, a progressive process $(\phi_t )$ is said to be a {\em strong supermartingale} if the associated family 
$(\phi_\theta, \theta\in \T)$ is a supermartingale family.\\
\indent For each progressive process $(\phi_t )$, the smallest strong supermartingale process greater or equal to $(\phi_t )$, 
when it exists, is called the {\em Snell envelope process} of $(\phi_t )$ and is denoted by $\hat{ \mathcal{R}}[(\phi_t)]$. 
By using Proposition \ref{prop.snell}, one can show that, for each progressive process $(\phi_t )$ satisfying $E\left[\esssup_{t\in[0, T]} \phi_t^-\right] <+\infty$
 and $\sup_{\theta\in \T} E[\phi_{\theta} ] < \infty$, the associated value function process $(v_t)$ aggregating 
the family $(v(\theta), \theta\in \T)$ (defined as in Proposition B.1. in \cite{KQ}), is equal to the Snell envelope process of $(\phi_t )$, that is, $(v_t)= \hat{ \mathcal{R}}[(\phi_t)]$. Equalities (\ref{eee}) thus lead to the equalities  $(J_t)=\hat{\mathcal{R}}[(J'_t+\xi_t)]$ and $(J'_t)=\hat{\mathcal{R}}[(J_t-\zeta_t)]$, which are well-known in the literature on the Dynkin game problem.

The above points underline again the relevance of the framework of admissible families to study the Dynkin game problem. 
\section{Complementary result} 
\begin{proposition}\label{ajoute} 
Suppose the families 
$J$ and $J'$ are uniformly integrable. Suppose also that the families $\xi$ and $\zeta$ are left-limited along stopping times with $\{\xi(\tau^-) = \zeta(\tau^-)\} = \emptyset$ a.s.\,, for each predictable stopping time $\tau $ $\in$ $\T$.
%Suppose that for each $\theta \in \T_0$ and for each non decreasing sequence of stopping times $(\theta_n)_{n \in \N}$ such that $ \theta^n \uparrow \theta$, we have
%\begin{equation} \label{SU}
%\{ \xi(\theta) = \zeta(\theta) \} \cap \{\theta_n<\theta,\,\, \mbox{for all}\,\,n\, \}= \emptyset \quad {\rm a.s.}
%\end{equation}
If $\xi$ and $-\zeta$ are right-USCE and strong left-USCE, then  the families $J$ and $J'$ are USCE and consequently,
for each $\theta \in \T$\,, the pair $\left(  \tau_{*} (\theta),\sigma_* (\theta) \right)$, defined by \eqref{tauetoile} and \eqref{sigmaetoile}, is a $\theta$-saddle point .
\end{proposition}

\begin{proof}The proof is the same as that of Theorem \ref{left-USCE} except the proof of equality (\ref{empty}) for a given stopping time $\theta$ $\in$ $\T$, which is slightly different. This equality actually follows from the same arguments as those used  at the end of Theorem \ref{left-USCE} on $C\cap \{\theta=T\}$.
%replaced
% By Lemma \ref{Propriete 3} and since $\xi$ and $\zeta$ are left-limited along stopping times, we have
%\begin{eqnarray*}
%E[ J(\theta^-){{\bf 1}}_C] &=&  E[ J'(\theta^-){{\bf 1}}_C]+ \sup_{\lambda \in ]0,1[} \limsup_{p \to \infty} E[\xi(\tau^{\lambda} (\theta_p)) {{\bf 1}}_{C}] \\
%&\leq & E[ J'(\theta^-){{\bf 1}}_{C}]+E[ \xi(\theta^-){{\bf 1}}_{C}],
%\end{eqnarray*}
%
%\begin{eqnarray*}
%E[ J'(\theta^-){{\bf 1}}_C] &=&  E[ J(\theta^-){{\bf 1}}_C]- \inf_{\lambda \in ]0,1[} \liminf_{p \to \infty} E[\zeta(\sigma^{\lambda} (\theta_p)) {{\bf 1}}_{C}] \\
%&\leq & E[ J(\theta^-){{\bf 1}}_{C}]-E[ \zeta(\theta^-){{\bf 1}}_{C}].
%\end{eqnarray*}
%%where the last inequality follows from the inequality 
%%$$\liminf_{p \to \infty}  E[\zeta( \sigma^{\lambda} (\theta_p)} \wedge \theta){{\bf 1}}_{C}  ]\geq E[\zeta(\theta) {{\bf 1}}_{C}  ],$$ due to the left-LSCE property of $\zeta$ in the sense (\ref{ }).
%By adding the two above inequalities, we get 
%\[0 \leq E[ \xi(\theta^-){{\bf 1}}_{C}]- E[ \zeta(\theta^-){{\bf 1}}_{C}],\]
%which leads to $\xi(\theta^-)=\zeta(\theta^-)$ a.s.\,on $C$. Since by assumption $\{ \xi(\theta^-) = \zeta(\theta^-) \} \cap A = \emptyset$ a.s.\, It follows that $P(C)=0$, which gives equality (\ref{empty}).
\end{proof}

\subsection*{Aknowledgement} The authors thank the anonymous Referee for his helpful comments and suggestions which have much contributed to the present version.

\end{document}